\documentclass[oneside,11pt]{amsart}
\usepackage{amsmath}
\usepackage{amssymb}
\usepackage{graphicx}
\usepackage{caption}
\usepackage{subcaption}
\usepackage{pinlabel}
\usepackage{tikz}

\newtheorem{thm}{Theorem}[section]
\newtheorem{prop}[thm]{Proposition}
\newtheorem{lem}[thm]{Lemma}

\theoremstyle{definition}
\newtheorem{defn}[thm]{Definition}
\newtheorem{rem}[thm]{Remark}

\newtheorem{exmp}[thm]{Example}
\newtheorem{ques}[thm]{Question}

\newcommand{\abs}[1]{\lvert{#1}\rvert}

\renewcommand{\bar}[1]{\overline{#1}}

\newcommand{\set}[2]{\{\,{#1} \mid {#2} \,\}}
\newcommand{\bigset}[2]{ \bigl\{ \, {#1} \bigm| {#2} \, \bigr\} }
\renewcommand{\emptyset}{\varnothing}

\newcommand{\field}[1]{\mathbb{#1}}

\newcommand{\PP}{\field{P}}

\newcommand{\RR}{\field{R}}

\renewcommand{\implies}{\Rightarrow}

\DeclareMathOperator{\CAT}{CAT}

\newcommand{\ball}[2]{B ( {#1}, {#2} )}

\newcommand{\nbd}[2]{\mathcal{N}_{#2}({#1})} 

\newcommand{\Set}[1]{\mathcal{#1}}

\DeclareMathOperator{\Div}{Div}

\DeclareMathOperator{\Cayley}{Cayley}

\begin{document}

\title[Divergence spectra and Morse boundaries]{Divergence spectra and Morse boundaries of relatively hyperbolic groups}

\author{Hung Cong Tran}
\address{Department of Mathematics\\
 The University of Georgia\\
1023 D. W. Brooks Drive\\
Athens, GA 30605\\
USA}
\email{hung.tran@uga.edu}

\date{\today}

\begin{abstract}
We introduce a new quasi-isometry invariant, called the divergence spectrum, to study finitely generated groups. We compare the concept of divergence spectrum with the other classical notions of divergence and we examine the divergence spectra of relatively hyperbolic groups. We show the existence of an infinite collection of right-angled Coxeter groups which all have exponential divergence but they all have different divergence spectra. We also study Morse boundaries of relatively hyperbolic groups and examine their connection with Bowditch boundaries.
\end{abstract}

\maketitle

\section{Introduction}
In \cite{MR1254309}, Gersten defined a quasi-isometry invariant of spaces, called divergence. He used divergence to classify certain 3--manifold groups up to quasi-isometry (see \cite{MR1302334}). The concept of divergence has also been studied by Macura \cite{MR3032700}, Behrstock-Charney \cite{MR2874959}, Duchin-Rafi \cite{MR2563768}, Dru{\c{t}}u-Mozes-Sapir \cite{MR2584607}, Sisto \cite{Sisto} and others. Since the concept of divergence is a quasi-isometry invariant, it is therefore a useful tool to study quasi-isometry classification of finitely generated groups. However, the concept of divergence fails to classify finitely generated groups in some certain cases. More precisely, we can not use this concept to decide whether two groups are quasi-isometric if they have the same divergence. For example, let $(G_1,H_1)$ and $(G_2,H_2)$ be two finitely presented one-ended relatively hyperbolic groups. Then, the divergence of $G_1$ and $G_2$ are both exponential (\cite{Sisto}). Therefore, we introduce a new quasi-isometry invariant, called the divergence spectrum, to deal with such circumstances. We remark that the idea of divergence spectra was originally suggested by Charney. 

We now sketch the idea of divergence spectra. Let $\gamma$ be a bi-infinite geodesic in a geodesic space $X$. The lower divergence of $\gamma$ is a function $g: (0,\infty)\to(0,\infty]$, where for each positive number $r$ the value $g(r)$ is the infimum on the lengths of all paths connecting some pair of points $x$, $y$ on $\gamma$ which lies outside the open ball with radius $r$ about the midpoint between $x$ and $y$ on $\gamma$. The divergence spectrum of $X$ is a family of functions that consists of all lower divergence functions of Morse bi-infinite geodesics in $X$. The divergence spectrum is a quasi-isometry invariant (see Section~ \ref{spectrum}) and therefore, we can use this concept to define the divergence spectrum of a finitely generated group via its Cayley graphs. If two spaces have different divergences, then they have different divergence spectra in most of known cases. Moreover, divergence spectra, in some certain cases, can tell us the difference between two finitely generated groups up to quasi-isometry although they have the same divergence. In this paper, we use the concept of divergence spectrum to classify a certain collection of relatively hyperbolic right-angled Coxeter groups up to quasi-isometry.

In Section \ref{sporh}, we study divergence spectra of relatively hyperbolic groups in relation to the geometric properties of their peripheral subgroups. In terms of divergence spectra, we show that there is a gap between the divergence of peripheral subgroups and the divergence of the whole relatively hyperbolic group (see the following theorem). 

\begin{thm}
\label{i1}
Let $(G,\PP)$ be a finitely generated relatively hyperbolic group. Suppose that there is a finite generating set $S$ such that for each subgroup $P$ in $\PP$ the set $S\cap P$ generates $P$ and the Cayley graph $\Gamma(P,S\cap P)$ of group $P$ with respect to the generating set $S\cap P$ is one-ended with the geodesic extension property. Let $f=\max\{Div_{\Gamma(P,S\cap P),e}\mid P\in\PP\}$. Then, the divergence spectrum of $G$ only contains functions that are at least exponential or dominated by $f$.
\end{thm}

We also show that the divergence spectrum of a peripheral subgroup of a relatively hyperbolic group is contained in the divergence spectrum of the whole group in most of cases (see the following theorem). 

\begin{thm}
\label{i2}
Let $(G,\PP)$ be a finitely generated relatively hyperbolic group. Let $P$ be a peripheral subgroup in $\PP$ such that the divergence spectrum of $P$ only contains subexponential functions. Then the divergence spectrum of $P$ is a subset of the divergence spectrum of $G$. 
\end{thm}


In Section \ref{spora}, we use divergence spectra to study a certain class of right-angled Coxeter groups. In \cite{MR3314816}, Dani-Thomas showed that for a polynomial of any degree, there is a right-angled Coxeter group with divergence equivalent to that polynomial. Moreover, the divergence of a right-angled Coxeter group was proved to be either exponential or at most polynomial and its divergence is exponential if and only if the group is relatively hyperbolic (see Behrstock-Hagen-Sisto \cite{BHSC}, which has an appendix written jointly with Caprace). By using properties of divergence spectra of relatively hyperbolic groups, we show the existence of an infinite collection of right-angled Coxeter groups which all have exponential divergence but are not pairwise quasi-isometric (see the following theorem).


\begin{figure}
\begin{tikzpicture}[scale=0.7]

\draw (-14,1) node[circle,fill,inner sep=1pt, color=black](1){} -- (-13,2) node[circle,fill,inner sep=1pt, color=black](1){}-- (-12,1) node[circle,fill,inner sep=1pt, color=black](1){}-- (-13,0) node[circle,fill,inner sep=1pt, color=black](1){} -- (-14,1) node[circle,fill,inner sep=1pt, color=black](1){}; 

\draw (-14,1) node[circle,fill,inner sep=1pt, color=black](1){} -- (-14,3) node[circle,fill,inner sep=1pt, color=black](1){} -- (-12,3) node[circle,fill,inner sep=1pt, color=black](1){} --(-12,1) node[circle,fill,inner sep=1pt, color=black](1){};

\node at (-13,4) {$\Omega_1$};

\draw (-11,1) node[circle,fill,inner sep=1pt, color=black](1){} -- (-10,2) node[circle,fill,inner sep=1pt, color=black](1){}-- (-9,1) node[circle,fill,inner sep=1pt, color=black](1){}-- (-10,0) node[circle,fill,inner sep=1pt, color=black](1){} -- (-11,1) node[circle,fill,inner sep=1pt, color=black](1){}; 

\draw (-10,2) node[circle,fill,inner sep=1pt, color=black](1){} -- (-8,1) node[circle,fill,inner sep=1pt, color=black](1){}-- (-10,0) node[circle,fill,inner sep=1pt, color=black](1){};

\draw (-11,1) node[circle,fill,inner sep=1pt, color=black](1){} -- (-10,-1) node[circle,fill,inner sep=1pt, color=black](1){}-- (-9,1) node[circle,fill,inner sep=1pt, color=black](1){};

\draw (-11,1) node[circle,fill,inner sep=1pt, color=black](1){} -- (-11,3) node[circle,fill,inner sep=1pt, color=black](1){}-- (-8,3) node[circle,fill,inner sep=1pt, color=black](1){} -- (-8,1) node[circle,fill,inner sep=1pt, color=black](1){};

\node at (-9.5,4) {$\Omega_2$};

\draw (-7,1) node[circle,fill,inner sep=1pt, color=black](1){} -- (-6,2) node[circle,fill,inner sep=1pt, color=black](1){}-- (-5,1) node[circle,fill,inner sep=1pt, color=black](1){}-- (-6,0) node[circle,fill,inner sep=1pt, color=black](1){} -- (-7,1) node[circle,fill,inner sep=1pt, color=black](1){}; 

\draw (-6,2) node[circle,fill,inner sep=1pt, color=black](1){} -- (-4,1) node[circle,fill,inner sep=1pt, color=black](1){}-- (-6,0) node[circle,fill,inner sep=1pt, color=black](1){};

\draw (-6,2) node[circle,fill,inner sep=1pt, color=black](1){} -- (-3,1) node[circle,fill,inner sep=1pt, color=black](1){}-- (-6,0) node[circle,fill,inner sep=1pt, color=black](1){};

\draw (-4,1) node[circle,fill,inner sep=1pt, color=black](1){} -- (-5,-1) node[circle,fill,inner sep=1pt, color=black](1){}-- (-6,-1) node[circle,fill,inner sep=1pt, color=black](1){};

\draw (-7,1) node[circle,fill,inner sep=1pt, color=black](1){} -- (-6,-1) node[circle,fill,inner sep=1pt, color=black](1){}-- (-5,1) node[circle,fill,inner sep=1pt, color=black](1){};

\draw (-7,1) node[circle,fill,inner sep=1pt, color=black](1){} -- (-7,3) node[circle,fill,inner sep=1pt, color=black](1){}-- (-3,3) node[circle,fill,inner sep=1pt, color=black](1){} -- (-3,1) node[circle,fill,inner sep=1pt, color=black](1){};

\node at (-5,4) {$\Omega_3$};

\draw (-2,1) node[circle,fill,inner sep=1.5pt, color=black](1){};
\draw (-1.5,1) node[circle,fill,inner sep=1.5pt, color=black](1){};
\draw (-1,1) node[circle,fill,inner sep=1.5pt, color=black](1){};
\draw (-0.5,1) node[circle,fill,inner sep=1.5pt, color=black](1){};
\draw (0,1) node[circle,fill,inner sep=1.5pt, color=black](1){};

\draw (1,1) node[circle,fill,inner sep=1pt, color=black](1){} -- (2,2) node[circle,fill,inner sep=1pt, color=black](1){}-- (3,1) node[circle,fill,inner sep=1pt, color=black](1){}-- (2,0) node[circle,fill,inner sep=1pt, color=black](1){} -- (1,1) node[circle,fill,inner sep=1pt, color=black](1){}; 

\draw (2,2) node[circle,fill,inner sep=1pt, color=black](1){} -- (4,1) node[circle,fill,inner sep=1pt, color=black](1){}-- (2,0) node[circle,fill,inner sep=1pt, color=black](1){};

\draw (2,2) node[circle,fill,inner sep=1pt, color=black](1){} -- (5,1) node[circle,fill,inner sep=1pt, color=black](1){}-- (2,0) node[circle,fill,inner sep=1pt, color=black](1){};

\draw (2,2) node[circle,fill,inner sep=1pt, color=black](1){} -- (9,1) node[circle,fill,inner sep=1pt, color=black](1){}-- (2,0) node[circle,fill,inner sep=1pt, color=black](1){};

\draw (2,2) node[circle,fill,inner sep=1pt, color=black](1){} -- (8,1) node[circle,fill,inner sep=1pt, color=black](1){}-- (2,0) node[circle,fill,inner sep=1pt, color=black](1){};

\draw (4,1) node[circle,fill,inner sep=1pt, color=black](1){} -- (3,-1) node[circle,fill,inner sep=1pt, color=black](1){}-- (2,-1) node[circle,fill,inner sep=1pt, color=black](1){};

\draw (5,1) node[circle,fill,inner sep=1pt, color=black](1){} -- (4,-1) node[circle,fill,inner sep=1pt, color=black](1){}-- (3,-1) node[circle,fill,inner sep=1pt, color=black](1){};

\draw (8,1) node[circle,fill,inner sep=1pt, color=black](1){} -- (7,-1) node[circle,fill,inner sep=1pt, color=black](1){}-- (6,-1) node[circle,fill,inner sep=1pt, color=black](1){};

\draw[densely dotted] (7,1) -- (6.5,0);

\draw (6.5,0) -- (6,-1);

\draw[densely dotted] (4,-1) node[circle,fill,inner sep=1pt, color=black](1){}-- (6,-1) node[circle,fill,inner sep=1pt, color=black](1){};

\draw[densely dotted] (5.5,1) -- (7,1) node[circle,fill,inner sep=1pt, color=black](1){};

\draw (1,1) node[circle,fill,inner sep=1pt, color=black](1){} -- (2,-1) node[circle,fill,inner sep=1pt, color=black](1){}-- (3,1) node[circle,fill,inner sep=1pt, color=black](1){};

\node at (2,-0.25) {$b_0$};

\node at (2,2.25) {$a_0$};

\node at (0.75,1) {$b_1$};

\node at (3.3,1) {$a_1$};

\node at (4.3,1) {$a_2$};

\node at (5.3,1) {$a_3$};

\node at (8.4,0.6) {$a_{d-1}$};

\node at (9.4,1) {$a_d$};

\node at (2,-1.4) {$b_2$};

\node at (3,-1.4) {$b_3$};

\node at (4,-1.4) {$b_4$};

\node at (6,-1.4) {$b_{d-1}$};

\node at (7,-1.4) {$b_d$};

\node at (1,3.4) {$c_1$};

\node at (9,3.4) {$c_2$};

\draw (1,1) node[circle,fill,inner sep=1pt, color=black](1){} -- (1,3) node[circle,fill,inner sep=1pt, color=black](1){}-- (9,3) node[circle,fill,inner sep=1pt, color=black](1){} -- (9,1) node[circle,fill,inner sep=1pt, color=black](1){};

\node at (5,4) {$\Omega_d$};

\end{tikzpicture}

\caption{}
\label{azero}
\end{figure}

\begin{thm}
\label{i3}
For each $d\geq 2$, let $\Omega_d$ be a graph in Figure \ref{azero} and $G_{\Omega_d}$ the associated right-angled Coxeter group. If $d_1$ and $d_2$ are two different positive integers, then $G_{\Omega_{d_1}}$ and $G_{\Omega_{d_2}}$ have different divergence spectra. Therefore, they are not quasi-isometric.
\end{thm}

We remark that each graph $\Omega_d$ is a modification of a graph $\Gamma_d$ in Figure 5.1 \cite{MR3314816}. Dani-Thomas \cite{MR3314816} built the graphs $\Gamma_d$ to study divergence of right-angled Coxeter groups. They proved that the divergence of $G_{\Gamma_d}$ is a polynomial $r^d$. Our infinite collection of right-angled Coxeter groups is a variation of the example created by them. Here we construct the graph $\Omega_d$ containing the subgraph $\Gamma_d$ in some certain way to create the relatively hyperbolic right-angled Coxeter group $G_{\Omega_d}$ with respect to the right-angled Coxeter subgroup $G_{\Gamma_d}$. We use the difference on divergence of peripheral subgroups $G_{\Gamma_d}$ proved by Dani-Thomas as one of the key points to show that all groups $G_{\Omega_d}$ have different divergence spectra. 

We also remark that there is an alternate (shorter) proof that does not use the divergence spectrum to differentiate groups $G_{\Omega_d}$ by using a combination of works in \cite{MR2501302}, \cite{MR3450952}, \cite{MR1254309}, \cite{Sisto}, and \cite{MR3314816} (see Remark \ref{aw}). However, Theorem \ref{i3} is a concrete illustration of how divergence spectra could be used to distinguish quasi-isometry classes of finitely generated groups. We hope that divergence spectra will prove useful for classifying more finitely generated groups up to quasi-isometry. 

In \cite{MC}, Cordes generalized the concept contracting boundary on a $\CAT(0)$ space (see \cite{MR3339446}) by defining a quasi-isometry invariant of proper geodesic metric spaces, called the Morse boundary. The Morse boundary of a space $X$, denoted $\partial_M X$, is the set of all Morse geodesic rays in $X$ where two geodesic rays $\alpha$, $\alpha'$ are equivalent if there exists a constant $K$ such that $d\bigl(\alpha(t), \alpha'(t)\bigr) < K$ for all $t>0$. Fix a basepoint $p$ and for each Morse gauge $N$, Cordes topologize the set $\partial_M^N X_p$ of all equivalent classes of all $N$--Morse geodesic rays as one does for the Gromov boundary of a hyperbolic space. He endows the Morse boundary with the topology of the direct limit over all Morse gauges and shows that this boundary is independent of basepoint. Moreover, he proves that Morse boundary is a quasi-isometry invariant and therefore gives a well-defined boundary for any finitely generated group.

When investigating the behaviors of Morse geodesics in Cayley graphs of finitely generated relatively hyperbolic groups, we obtain some results on their Morse boundaries. We first show the connection between the Morse boundary of each peripheral subgroup in a finitely generated relatively hyperbolic group with the Morse boundary of the whole group (see the following theorem).

\begin{thm}
\label{i4}
Let $(G,\PP)$ be a finitely generated relatively hyperbolic group. Then for each peripheral subgroup $P$ in $\PP$ the inclusion $i_P\!:P\hookrightarrow G$ induces a Morse preserving map. Therefore, $\partial_M i_P: \partial_M P \rightarrow \partial_M G$ is a topological embedding.
\end{thm}

We also show a connection between Morse boundaries and Bowditch boundaries defined in \cite{MR2922380}.

\begin{thm}
\label{i5}
Let $(G,\PP)$ be a finitely generated relatively hyperbolic group. Then there is a $G$--equivariant continuous map $f$ from the Morse boundary $\partial_M G$ to the Bowditch boundary $\partial (G,\PP)$ with the following properties:
\begin{enumerate}
\item The map $f$ maps the set of non-peripheral limit points of $\partial_M G$ injectively into the set of non-parabolic points of $\partial (G,\PP)$.
\item The map $f$ maps peripheral limit points of the same type in $\partial_M G$ to the same parabolic point in $\partial (G,\PP)$.
\end{enumerate}
In particular, if the Morse boundary of each peripheral subgroup is empty, then the maps $f$ maps the Morse boundary $\partial_M G$ injectively into the set of non-parabolic points of $\partial (G,\PP)$.
\end{thm}


The outline of the paper is as follows. In Section~ \ref{prelim}, we prepare some preliminary knowledge for the main part of the paper. This knowledge will be used to define the divergence spectrum and compute divergence spectra of certain groups. In Section~ \ref{spectrum}, we give the precise definition of the divergence spectrum of a geodesic space and use this concept to define the divergence spectrum of a finitely generated group. In Section~ \ref{conex}, we review other concepts of divergence and show some connections between the divergence spectrum and these concepts. We also give some examples of the divergence spectra of some spaces and finitely generated groups. In Section~ \ref{sporh}, we examine some properties of divergence spectra of relatively hyperbolic groups. In this section, readers can find the proof of Theorem \ref{i1} and Theorem \ref{i2}. In Section~ \ref{spora}, we use divergence spectra to show the existence of an infinite collection of right-angled Coxeter groups which all have exponential divergence but are not pairwise quasi-isometric. In this section, readers can find the proof of Theorem \ref{i3}. In Section~ \ref{mbdr}, we review the concept of Morse boundary in \cite{MC} and give the proof of Theorem \ref{i4}. In Section~ \ref{bbvsmb}, we review the concept of Bowditch boundary in \cite{MR2922380} and show a connection between Morse boundary and Bowditch boundary for a finitely generated relatively hyperbolic group $(G,\PP)$. In this section, readers can find the proof of Theorem \ref{i5}.

\subsection*{Acknowledgments}
I would like to thank Prof.~Ruth Charney for her suggestion to study the quasi-isometry invariant divergence spectra and her encouragement to publish this paper. I also thank Prof.~Kim Ruane for helpful conversations about the connection between contracting boundaries and Bowditch boundaries that gave me motivation for studying Morse boundaries of relatively hyperbolic groups. I want to thank Prof.~Christopher Hruska, Prof.~Pallavi Dani, Hoang Thanh Nguyen and Kevin Schreve for their very helpful conversations and suggestions. I also thank the referee for advice that improved the exposition of the paper.

\section{Preliminaries}
\label{prelim}
In this section, we discuss some preliminary background before discussing the main part of the paper. We first construct the notions of domination and equivalence. We review some concepts in geometric group theory: geodesic spaces, quasi-geodesics, Morse quasi-geodesic, quasi-isometry, quasi-isometric embedding, and the geodesic extension property.

\begin{defn}
Let $\mathcal{M}$ be the collection of all functions from $[0,\infty)$ to $[0,\infty]$. Let $f$ and $g$ be arbitrary elements of $\mathcal{M}$. \emph{The function $f$ is dominated by the function $g$}, denoted \emph{$f\preceq g$}, if there are positive constants $A$, $B$, $C$ and $D$ such that $f(x)\leq Ag(Bx)+Cx$ for all $x>D$. Two function $f$ and $g$ are \emph{equivalent}, denoted \emph{$f\sim g$}, if $f\preceq g$ and $g\preceq f$. \emph{The function $f$ is strictly dominated by the function $g$}, denoted \emph{$f\prec g$}, if $f$ is dominated by $g$ and they are not equivalent.
\end{defn}

\begin{rem}
The relations $\preceq$ and $\prec$ are transitive. The relation $\sim$ is an equivalence relation on the set $\mathcal{M}$.

Let $f$ and $g$ be two polynomial functions in the family $\mathcal{M}$. We observe that $f$ is dominated by $g$ iff the degree of $f$ is less than or equal to the degree of $g$ and they are equivalent iff they have the same degree. All exponential functions of the form $a^{bx+c}$, where $a>1, b>0$ are equivalent. Therefore, a function $f$ in $\mathcal{M}$ is \emph{linear, quadratic or exponential...} if $f$ is respectively equivalent to any polynomial with degree one, two or any function of the form $a^{bx+c}$, where $a>1, b>0$.

\end{rem}

\begin{defn}
Let $X$ be a geodesic space and $A$ a subspace of $X$. Let $r$ be any positive number.
\begin{enumerate}
\item $N_r(A)=\bigset{x \in X}{d_X(x, A)<r}$
\item $\partial N_r(A)=\bigset{x \in X}{d_X(x, A)=r}$ 
\item $C_r(A)=X-N_r(A)$.
\end{enumerate}
\end{defn}

\begin{defn}
Let $(X,d_X)$ and $(Y,d_Y)$ be two metric spaces. A map $\Phi$ from $X$ to $Y$ is a \emph{$(K,L)$--quasi-isometric embedding} if for all $x_1, x_2$ in $X$ the following inequality holds:\[({1}/{K})\,d_X(x_1,x_2)-L\leq d_Y\bigl(\Phi(x_1),\Phi(x_2)\bigr)\leq K\,d_X(x_1,x_2)+L.\]
If, in addition, $N_L\bigl(\Phi(X)\bigr)=Y$, then $\Phi$ is called a \emph{$(K,L)$--quasi-isometry}. Two spaces $X$ and $Y$ are quasi-isometric if there is a $(K,L)$--quasi-isometry from $X$ to $Y$.

The special case of a quasi-isometric embedding where the domain is a connected interval in $\RR$ (possibly all of $\RR$) is called a \emph{$(K,L)$--quasi-geodesic}. A \emph{geodesic} is a $(1,0)$--quasi-geodesic. The metric space $X$ is a \emph{geodesic space} if any pair of points in $X$ can be joined by a geodesic segment.
\end{defn}

\begin{rem}
We assume that all metric spaces in this paper are proper geodesic metric spaces (i.e. every closed ball is compact).
\end{rem}

\begin{defn}
A space $X$ has the \emph{geodesic extension property} if any geodesic segment lies in a bi-infinite geodesic in $X$. 
\end{defn}


\begin{defn}
A quasi-geodesic $\gamma$ is \emph{$M$--Morse} if for any constants $K\geq 1$ and $L> 0$, there is a constant $M = M(K,L)$ such that every $(K,L)$--quasi-geodesic $\sigma$ with endpoints on $\gamma$ lies in the $M$--neighborhood of $\gamma$. A quasi-geodesic $\gamma$ is \emph{Morse} if it is $M$--Morse for some $M$. We call $M$ a \emph{Morse gauge} for $\gamma$.
\end{defn}

We now come up with some lemmas that prepare us to define a new quasi-isometry invariant, called the divergence spectrum. The proofs of the following two lemmas are obvious, and we leave them to the reader.

\begin{lem}
\label{b2}
\label{qip}
Let $(X,d_X)$ and $(Y,d_Y)$ be two geodesic spaces and the map $\Phi$ from $X$ to $Y$ a $(K,L)$--quasi-isometry. Then there is a constant $C=C(K,L)\geq1$ such that the following hold:
\begin{enumerate}
\item $({1}/{C})\,d_X(x_1,x_2)-1\leq d_Y\bigl(\Phi(x_1),\Phi(x_2)\bigr)\leq C\,d_X(x_1,x_2)+C$, for all $x_1, x_2$ in $X$
\item $N_C\bigl(\Phi(X)\bigr)=Y$
\item If $\alpha$ is a path connecting two points $x_1$ and $x_2$ in $X$, then there is a path $\beta$ connecting $\Phi(x_1)$ and $\Phi(x_2)$ in $Y$ such that the Hausdorff distance between $\Phi(\alpha)$ and $\beta$ is at most $C$. Moreover, $\abs{\beta}\leq C\abs{\alpha}+C$.
\item If $\beta$ is a path connecting two points $\Phi(x_1)$ and $\Phi(x_2)$ for some $x_1, x_2 \in X$, then there is a path $\alpha$ connecting $x_1$ and $x_2$ in $X$ such that the Hausdorff distance between $\Phi(\alpha)$ and $\beta$ is at most $C$. Moreover, $\abs{\alpha}\leq C\abs{\beta}+C$.
\end{enumerate} 
\end{lem}

\begin{lem}
\label{morse2}
Let $\gamma$ be a Morse quasi-geodesic. For any constants $K\geq 1$ and $L> 0$, there is a constant $M = M(K,L)$ such that the following hold. Let $\sigma$ be a $(K,L)$--quasi-geodesic with endpoints $\gamma(t_1)$ and $\gamma(t_2)$ on $\gamma$ ($t_1<t_2$). The Hausdorff distance between $\sigma$ and $\gamma\bigl([t_1,t_2]\bigr)$ is bounded above by $M$.
\end{lem}

\begin{lem}
\label{b1}
Let $\Phi$ be a $(K,L)$--quasi-isometry from $X$ to $Y$ and $\gamma$ an \emph{$M$--Morse} bi--infinite quasi-geodesic in $X$. 
\begin{enumerate}
\item If $\rho$ is a bi-infinite quasi-geodesic whose Hausdorff distance from $\gamma$ is at most $C$, then $\rho$ is $M_1$--Morse, where $M_1$ depends only on $M$ and $C$.
\item A bi-infinite quasi-geodesic $\Phi\circ\gamma$ is $M_2$--Morse, where $M_2$ depends only on $K$, $L$ and $M$. 
\item There is a $M_3$--Morse bi-infinite geodesic $\beta$ in $X$ such that the Hausdorff distance between $\gamma$ and $\beta$ is finite.
\end{enumerate} 
\end{lem}
\begin{proof}
We obtain (1) and (2) from Lemma 2.5 in \cite{MR3339446} We only need to prove the statement (3). By Lemma \ref{morse2}, there is a constant $D>0$ such that the following holds. Let $\sigma$ be the geodesic with endpoints $\gamma(t_1)$, $\gamma(t_2)$ on $\gamma$ ($t_1<t_2$). Then the Hausdorff distance between $\sigma$ and $\gamma\bigl([t_1,t_2]\bigr)$ is bounded above $D$. For each $n$, let $\beta_n$ be geodesic segment from $\gamma(-n)$ to $\gamma(n)$. Then the Hausdorff distance between $\beta_n$ and $\gamma\bigl([-n,n]\big)$ is bounded above by $D$. Since $X$ is proper, then there is a bi-infinite geodesic $\beta$ such that the Hausdorff distance between $\beta$ is bounded above by $D+1$. The Morse property of $\beta$ is obtained from (1) easily.
\end{proof}

\section{The divergence spectrum}
\label{spectrum}
In this section, we introduce the concept of divergence spectra of geodesic spaces as well as finitely generated groups. We also prove that the divergence spectrum is a quasi-isometry invariant.
\begin{defn}[Lower divergence]
Let $\alpha$ be a bi-infinite geodesic. For any $r>0$ and $t\in \RR$, if there is no path from $\alpha(t-r)$ to $\alpha(t+r)$ that lies outside the open ball of radius $r$ about $\alpha(t)$, we define $\rho_{\alpha}(r,t)=\infty$. Otherwise, we let $\rho_{\alpha}(r,t)$ denote the infimum of the lengths of all paths from $\alpha(t-r)$ to $\alpha(t+r)$ that lies outside the open ball of radius $r$ about $\alpha(t)$. Define the \emph{lower divergence} of $\alpha$ to be the growth rate of the following function: \[ldiv_{\alpha}(r)=\inf_{t\in \RR}\rho_{\alpha}(r,t)\] 
\end{defn}

\begin{rem}
We remark that the concept of geodesic lower divergence was first introduced by Charney--Sultan \cite{MR3339446} to study contracting boundaries of $\CAT(0)$ spaces. We now use the geodesic lower divergence concept to define a new quasi-isometry invariant, called the divergence spectrum.
\end{rem}

\begin{defn}[The divergence spectrum]
The \emph{divergence spectrum} of a geodesic space $X$, denoted $S_X$, is a collection of functions from $[0, \infty)$ to $[0, \infty]$ such that a function $f$ belongs to $S_X$ if there is a Morse bi-infinite geodesic $\gamma$ with the lower divergence function equivalent to $f$.
\end{defn}

\begin{thm}
If $X$ and $Y$ are two quasi-isometric spaces, then they have the same divergence spectrum.
\end{thm}

\begin{proof}
Let $S_X$ and $S_Y$ be divergence spectra of $X$ and $Y$ respectively. We only need to prove $S_X\subset S_Y$ and the proof for the opposite direction is almost identical. Let $\Phi\!:X \to Y$ be a $(K,L)$--quasi-isometry. Let $f$ be any function in $S_X$. Then there is a Morse bi-infinite geodesic $\alpha$ in $X$ with the lower divergence equivalent to $f$. By Lemma \ref{b1}, there is a number $D>0$ and a Morse bi-infinite geodesic $\beta$ in $Y$ such that the Hausdorff distance between $\Phi\circ\alpha$ and $\beta$ is less than $D$. We need to prove that the lower divergence of $\beta$ is equivalent to $f$. The following two lemmas complete the proof of this theorem.
\end{proof}

\begin{lem}
The lower divergence function of $\beta$ is dominated by the lower divergence function of $\alpha$.
\end{lem}

\begin{proof}
Let $C=C(K,L)$ be the constant in Lemma \ref{b2}. Let $D_1=2C^2(D+1)+C+D+1$, $D_2=C+2D+D_1+1$ and $M=4C$. We are going to prove that for each $r>D_2$ \[ldiv_{\beta}(r)\leq C~ldiv_{\alpha}(Mr)+(2CM+3)r.\]

For each $t\in\RR$, let $\gamma$ be an arbitrary path from $\alpha(t-Mr)$ to $\alpha(t+Mr)$ that lies outside the open ball of radius $Mr$ about $\alpha(t)$. Since the Hausdorff distance between $\Phi\circ\alpha$ and $\beta$ is less than $D$, then there are $s_1$, $s_2$ in $\RR$ such that $d_Y\bigl(\Phi\circ\alpha(t-Mr), \beta(s_1)\bigr)<D$ and $d_Y\bigl(\Phi\circ\alpha(t+Mr), \beta(s_2)\bigr)<D$. Let $I$ be a subinterval of $\RR$ with endpoints $s_1$ and $s_2$. We need to show that there is an $s\in I$ such that $d_Y\bigl(\Phi\circ\alpha(t), \beta(s)\bigr)<D_1$.

We subdivide $I$ into m subintervals with lengths less than 1 by $(m+1)$ numbers as follows:
\[s_1=w_0<w_1<\cdots<w_m=s_2 \text{ if } s_1<s_2\] 
or \[s_2=w_0<w_1<\cdots<w_m=s_1 \text{ if } s_2<s_1.\] 
Since the Hausdorff distance between $\Phi\circ\alpha$ and $\beta$ is less than $D$, then there is some $w'_i$ in $\RR$ such that $d_Y\bigl(\Phi\circ\alpha(w'_i), \beta(w_i)\bigr)<D$. We choose $w'_0=t-Mr$, $w'_m=t+Mr$ if $s_1<s_2$ and $w'_0=t+Mr$, $w'_m=t-Mr$ if $s_2<s_1$. Let $I_i$ be a subinterval of $\RR$ with end points $w'_{i-1}$ and $w'_i$. Obviously, $[t-Mr, t+Mr] \subset \bigcup I_i$. Thus, $t$ lies in some $I_i$. Therefore, \begin{align*} d_Y\bigl(\Phi\circ\alpha(t), \beta(w_i)\bigr)& \leq d_Y\bigl(\Phi\circ\alpha(t), \Phi\circ\alpha(w'_i)\bigr)+d_Y\bigl(\Phi\circ\alpha(w'_i), \beta(w_i)\bigr) \\&\leq Cd_X\bigl(\alpha(t), \alpha(w'_i)\bigr)+C+D\\&\leq C\abs{t-w'_i}+C+D\\&\leq C\abs{w'_i-w'_{i-1}}+C+D.\end{align*}
Also, \begin{align*} \abs{w'_i-w'_{i-1}}&= d_X\bigl(\alpha(w'_i), \alpha(w'_{i-1})\bigr)\\&\leq Cd_Y\bigl(\Phi\circ\alpha(w'_i), \Phi\circ\alpha(w'_{i-1})\bigr)+C\\&\leq C[(d_Y\bigl(\Phi\circ\alpha(w'_i), \beta(w_i)\bigr)+d_Y\bigl(\beta(w_i), \beta(w_{i-1})\bigr)+d_Y\bigl(\beta(w_{i-1}), \Phi\circ\alpha(w'_{i-1})\bigr)]+C\\& \leq C(D+1+D)+C\\& \leq 2C(D+1) \end{align*}
Thus, $d_Y\bigl(\Phi\circ\alpha(t), \beta(w_i)\bigr)\leq 2C^2(D+1)+C+D<D_1$. Let $s=w_i$ and we now show that $3r\leq \abs{s-s_1}\leq (CM+1)r$.
In fact, \begin{align*} \abs{s-s_1}&= d_Y\bigl(\beta(s), \beta(s_1)\bigr)\\&\leq d_Y\bigl(\beta(s), \Phi\circ\alpha(t)\bigr)+d_Y\bigl(\Phi\circ\alpha(t), \Phi\circ\alpha(t-Mr)\bigr)+d_Y\bigl(\Phi\circ\alpha(t-Mr), \Phi\circ\alpha(s_1)\bigr)\\&\leq D_1+CMr+C+D\\&\leq (CM+1)r \end{align*}
and \begin{align*} \abs{s-s_1}&= d_Y\bigl(\beta(s), \beta(s_1)\bigr)\\&\geq d_Y\bigl(\Phi\circ\alpha(t),\Phi\circ\alpha(t-Mr)\bigr)-d_Y\bigl(\Phi\circ\alpha(t), \beta(s)\bigr)-d_Y\bigl(\Phi\circ\alpha(t-Mr), \Phi\circ\alpha(s_1)\bigr)\\&\geq \frac{Mr}{C}-1-D_1-D\\&\geq 4r-1-D_1-D\geq 3r \end{align*}

Similarly, $3r\leq \abs{s-s_2}\leq (CM+1)r$.

By Lemma \ref{b2}, there is a path from $\Phi\circ\alpha(t-Mr)$ to $\Phi\circ\alpha(t+Mr)$ such that $\abs{\gamma_1}\leq C\abs{\gamma}+C$ and the Hausdorff distance between $\Phi\circ\gamma$ and $\gamma_1$ is less than $C$. For each $y\in \gamma_1$, there is some $x\in \gamma$ such that $d_Y\bigl(y, \Phi(x)\bigr)<C$. Therefore, \begin{align*} d_Y\bigl(y, \beta(s)\bigr)& \geq d_Y\bigl(\Phi(x), \Phi\circ\alpha(t)\bigr)-d_Y\bigl(\Phi(x),y\bigr)-d_Y\bigl(\beta(s),\Phi\circ\alpha(t)\bigr) \\&\geq \frac{1}{C}d_X\bigl(x, \alpha(t)\bigr)-1-C-D_1\\&\geq \frac{Mr}{C}-C-D_1-1\\&\geq 4r-C-D_1-1> 3r.\end{align*}
Thus, $\gamma_1$ lies outside the open ball with radius $3r$ about $\beta(s)$. Let $\gamma_2$ be the geodesic connecting $\Phi\circ\alpha(t-Mr)$ and $\beta(s_1)$, then $\abs{\gamma_2}\leq D$ and $\gamma_2$ lie outside the open ball with radius $2r$ about $\beta(s)$. Let $\gamma_3$ be the geodesic connecting $\Phi\circ\alpha(t+Mr)$ and $\beta(s_2)$, then $\abs{\gamma_3}\leq D$ and $\gamma_3$ lies outside the open ball with radius $2r$ about $\beta(s)$. Let $s_3$ and $s_4$ be two different points in $I$ such that $\abs{s-s_3}=\abs{s-s_4}=r$. We choose $s_3$ lies between $s$, $s_1$ and $s_4$ lies between $s$, $s_2$. Since $\abs{s-s_1}\leq (CM+1)r$, then there is a geodesic segment $\gamma_4$ from $\beta(s_1)$ to $\beta(s_3)$ that lies outside the open ball with radius $r$ about $\beta(s)$ and the $\abs{\gamma_4}\leq (CM+1)r$. Since $\abs{s-s_2}\leq (CM+1)r$, then there is a geodesic segment $\gamma_5$ from $\beta(s_2)$ to $\beta(s_4)$ that lies outside the open ball with radius $r$ about $\beta(s)$ and the $\abs{\gamma_5}\leq (CM+1)r$. Let $\gamma_6=\gamma_4\cup\gamma_2\cup\gamma_1\cup \gamma_3\cup \gamma_5$. Then, $\gamma_6$ is a path from $\beta(s_3)$ to $\beta(s_4)$ that lies outside the open ball with radius $r$ about $\beta(s)$.

Moreover, \[\abs{\gamma_6}\leq 2(CM+1)r+2D+C\abs{\gamma}+C\leq C\abs{\gamma}+(2CM+3)r.\]
Therefore, \[ldiv_{\beta}(r)\leq C\abs{\gamma}+(2CM+3)r.\] 
Since $\gamma$ is an arbitrary path from $\alpha(t-Mr)$ to $\alpha(t+Mr)$ that lies outside the open ball with radius $r$ about $\alpha(t)$, then 
\[ldiv_{\beta}(r)\leq C~\rho_{\alpha}(Mr,t)+(2CM+3)r.\]
Since $t$ is an arbitrary number and $ldiv_{\alpha}(Mr)=\inf_{t\in\RR}\rho_{\alpha}(Mr,t)$, then 
\[ldiv_{\beta}(r)\leq C~ldiv_{\alpha}(Mr)+(2CM+3)r.\] 
Therefore, \[ldiv_{\beta}\preceq ldiv_{\alpha}.\]
\end{proof}

\begin{lem}
\label{lm1}
The lower divergence function of $\alpha$ is dominated by the lower divergence function of $\beta$. 
\end{lem}

\begin{proof}
Let $C=C(K,L)$ be the constant in Lemma \ref{b2}. Let $D_1=2C+3D+1$, $D_2=2CD+C+D+D_1+2$ and $M=4C$. We are going to prove that for each $r>D_2$ \[ldiv_{\alpha}(r)\leq C~ldiv_{\beta}(Mr)+(2CM+C+2)r.\]

For each $s\in\RR$, let $\gamma$ be an arbitrary path from $\beta(s-Mr)$ to $\beta(s+Mr)$ that lies outside the open ball of radius $Mr$ about $\beta(s)$. Since the Hausdorff distance between $\Phi\circ\alpha$ and $\beta$ is less than $D$, then there are $t_1$, $t_2$ in $\RR$ such that $d_Y\bigl(\Phi\circ\alpha(t_1), \beta(s-Mr)\bigr)<D$ and $d_Y\bigl(\Phi\circ\alpha(t_2), \beta(s+Mr)\bigr)<D$. Let $J$ be a subinterval of $\RR$ with endpoints $t_1$ and $t_2$. We need to show that there is a $t\in J$ such that $d_Y\bigl(\Phi\circ\alpha(t), \beta(s)\bigr)<D_1$.

We subdivide $J$ into m subintervals with lengths less than 1 by $(m+1)$ numbers as follows:
\[t_1=w_0<w_1<\cdots<w_m=t_2 \text{ if } t_1<t_2\] 
or \[t_2=w_0<w_1<\cdots<w_m=t_1 \text{ if } t_2<t_1.\] 
Since the Hausdorff distance between $\Phi\circ\alpha$ and $\beta$ is less than $D$, then there is some $w'_i$ in $\RR$ such that $d_Y\bigl(\Phi\circ\alpha(w_i), \beta(w'_i)\bigr)<D$. We choose $w'_0=s-Mr$, $w'_m=s+Mr$ if $t_1<t_2$ and $w'_0=s+Mr$, $w'_m=s-Mr$ if $t_2<t_1$. Let $J_i$ be a subinterval of $\RR$ with end points $w'_{i-1}$ and $w'_i$. Obviously, $[s-Mr, s+Mr] \subset \bigcup J_i$. Thus, $s$ lies in some $J_i$. Therefore, \begin{align*} d_Y\bigl(\beta(s), \Phi\circ\alpha(w_i)\bigr)& \leq d_Y\bigl(\beta(s), \beta(w'_i)\bigr)+d_Y\bigl(\beta(w'_i), \Phi\circ\alpha(w_i)\bigr) \\&\leq \abs{s-w'_i}+D \\&\leq \abs{w'_i-w'_{i-1}}+D.\end{align*}

Also, \begin{align*} \abs{w'_i-w'_{i-1}}&= d_Y\bigl(\beta(w'_i), \beta(w'_{i-1})\bigr)\\&\leq d_Y\bigl(\beta(w'_i), \Phi\circ\alpha(w_i)\bigr)+d_Y\bigl(\Phi\circ\alpha(w_i), \Phi\circ\alpha(w_{i-1})\bigr)+d_Y\bigl(\Phi\circ\alpha(w_{i-1}), \beta(w'_{i-1})\bigr)\\&\leq D+C\abs{w_i-w_{i-1}}+C+D\leq 2C+2D\end{align*}
Thus, $d_Y\bigl(\beta(s), \Phi\circ\alpha(w_i)\bigr)\leq 2C+3D<D_1$. Let $t=w_i$ and we now show that $3r\leq \abs{t-t_1}\leq (CM+C+1)r$. In fact, \begin{align*} \abs{t-t_1}&= d_X\bigl(\alpha(t), \alpha(t_1)\bigr)\\&\leq Cd_Y\bigl(\Phi\circ\alpha(t),\Phi\circ\alpha(t_1)\bigr)+C\\&\leq C\biggl(d_Y\bigl(\Phi\circ\alpha(t), \beta(s)\bigr)+d_Y\bigl(\beta(s), \beta(s-Mr)\bigr)+d_Y\bigl(\beta(s-Mr),\Phi\circ\alpha(t_1)\bigr)\biggr)+C\\&\leq C(D_1+Mr+D)+C\\&\leq C(M+1)r+r\leq (CM+C+1)r 
\end{align*}
and \begin{align*} \abs{t-t_1}&= d_X\bigl(\alpha(t), \alpha(t_1)\bigr)\\&\geq \frac{1}{C}d_Y\bigl(\Phi\circ\alpha(t),\Phi\circ\alpha(t_1)\bigr)-1\\&\geq \frac{1}{C}\biggl(d_Y\bigl(\beta(s), \beta(s-Mr)\bigr)-d_Y\bigl(\beta(s),\Phi\circ\alpha(t)\bigr)-d_Y\bigl(\beta(s-Mr),\Phi\circ\alpha(t_1)\bigr)\biggr)-1\\&\geq \frac{1}{C}(Mr-D_1-D)-1\\&\geq \frac{Mr}{C}-\frac{D_1+D}{C}-1\\&\geq 4r-\frac{D_1+D}{C}-1\geq 3r \end{align*}

Similarly, $3r\leq \abs{t-t_2}\leq (CM+C+1)r$.

Let $\gamma_2$ be the geodesic connecting $\Phi\circ\alpha(t_1)$ and $\beta(s-Mr)$, then $\abs{\gamma_2}\leq D$ and $\gamma_2$ lies outside the open ball with radius $Mr-D$ about $\beta(s)$. Let $\gamma_3$ be the geodesic connecting $\Phi\circ\alpha(t_2)$ and $\beta(s+Mr)$, then $\abs{\gamma_3}\leq D$ and $\gamma_3$ lies outside the open ball with radius $Mr-D$ about $\beta(s)$. Let $\gamma_4=\gamma_2\cup\gamma\cup\gamma_3$. Then, $\gamma_4$ is a path from $\Phi\circ\alpha(t_1)$ to $\Phi\circ\alpha(t_2)$ that lies outside the open ball with radius $Mr-D$ about $\beta(s)$.

By Lemma \ref{b2}, there is a path $\gamma_1$ from $\alpha(t_1)$ to $\alpha(t_2)$ such that $\abs{\gamma_1}\leq C\abs{\gamma_4}+C$ and the Hausdorff distance between $\Phi\circ\gamma_1$ and $\gamma_4$ is less than $C$. For each $x\in\gamma_1$, there is $y\in\gamma_4$ such that $d_Y\bigl(\Phi(x),y\bigr)<C$. Therefore, \begin{align*} d_X\bigl(x, \alpha(t)\bigr)& \geq \frac{1}{C}d_Y\bigl(\Phi(x), \Phi\circ\alpha(t)\bigr)-1\\&\geq \frac{1}{C}\biggl(d_Y\bigl(y, \beta(s)\bigr)-d_Y\bigl(y,\Phi(x)\bigr)-d_Y\bigl(\beta(s),\Phi\circ\alpha(t)\bigr)\biggr)-1\\&\geq \frac{1}{C}(Mr-D-C-D_1)-1\\&\geq \frac{Mr}{C}-\frac{D+D_1}{C}-2> 3r.\end{align*}
Thus, $\gamma_1$ lies outside the open ball with radius $3r$ about $\alpha(t)$. Let $t_3$ and $t_4$ be two different points in $J$ such that $\abs{t-t_3}=\abs{t-t_4}=r$. We choose $t_3$ lies between $t$, $t_1$ and $t_4$ lies between $t$, $t_2$. Since $\abs{t-t_1}\leq (CM+C+1)r$, then there is a geodesic segment $\gamma_5$ from $\alpha(t_1)$ to $\alpha(t_3)$ that lies outside the open ball with radius $r$ about $\alpha(t)$ and the $\abs{\gamma_5}\leq (CM+C+1)r$. Since $\abs{t-t_2}\leq (CM+C+1)r$, then there is a geodesic segment $\gamma_6$ from $\alpha(t_2)$ to $\alpha(t_4)$ that lies outside the open ball with radius $r$ about $\alpha(t)$ and the $\abs{\gamma_6}\leq (CM+C+1)r$. Let $\gamma_7=\gamma_5\cup\gamma_1\cup\gamma_6$. Then, $\gamma_7$ is a path from $\alpha(t_3)$ to $\alpha(t_4)$ that lies outside the open ball with radius $r$ about $\alpha(t)$.

Moreover, \begin{align*}\abs{\gamma_7}&\leq 2(CM+C+1)r+C\abs{\gamma_4}+C\\&\leq C\bigl(\abs{\gamma}+2D\bigr)+2(CM+C+1)r+C\\&\leq C\abs{\gamma}+2(CM+C+2)r.\end{align*}
Therefore, \[ldiv_{\alpha}(r)\leq C\abs{\gamma}+2(CM+C+1)r.\]
Since $\gamma$ is an arbitrary path from $\beta(s-Mr)$ to $\beta(s+Mr)$ that lies outside the open ball with radius about $\beta(s)$, then 
\[ldiv_{\alpha}(r)\leq C~\rho_{\beta}(Mr,s)+2(CM+C+1)r.\]
Since $s$ is an arbitrary number greater and $ldiv_{\beta}(Mr)=\inf_{s\in\RR}\rho_{\beta}(Mr,s)$, then 
\[ldiv_{\alpha}(r)\leq C~ldiv_{\beta}(Mr)+2(CM+C+1)r.\]
Thus, $ldiv_{\alpha}\preceq ldiv_{\beta}$.
\end{proof}

We are now ready to define the divergence spectrum of a finitely generated group.
\begin{defn} 
Let $G$ be a finitely generated group. We define \emph{the divergence spectrum} of $G$, denoted \emph{$S_G$} to be the divergence spectrum of the Cayley graph $\Gamma(G,S)$ for some finite generating set $S$.
\end{defn}

\section{Some connections between the divergence spectrum and other notions of divergence and examples}
\label{conex}
In this section we review some concepts of divergence of two geodesic rays, divergence of bi-infinite geodesics, divergence of geodesic spaces. We show the connection between the divergence spectrum and these concepts. We also give some examples of divergence spectra of some spaces as well as finitely generated groups.

We first recall the concept of the divergence of a pair of geodesic rays, the divergence of a bi-infinite geodesic and state some clear connections between these concepts and lower divergence.
\begin{defn}
The \emph{divergence} of two geodesic rays $\alpha$ and $\beta$ with the same initial point $x_0$ in a geodesic space $X$, denoted $\Div_{\alpha,\beta}$, is a function $g: (0,\infty)\to(0,\infty]$ defined as follows. For each positive $r$, if there is no path outside the open ball with radius $r$ about $x_0$ connecting $\alpha(r)$ and $\beta(r)$, we define $g(r)=\infty$. Otherwise, we define $g(r)$ to be the infimum on the lengths of all paths outside the open ball with radius $r$ about $x_0$ connecting $\alpha(r)$ and $\beta(r)$.

The \emph{divergence of a bi-infinite geodesic} $\gamma$, denoted $\Div_{\gamma}$, is the divergence of the two geodesic rays obtained from $\gamma$ with the initial point $\gamma(0)$.
\end{defn}

\begin{rem}
Let $\alpha$ be a bi-infinite geodesic in a geodesic space $X$. It is not hard to see that $ldiv_{\alpha}\leq Div_{\alpha}$. Moreover, $ldiv_{\alpha}\sim Div_{\alpha}$ if $\alpha$ is a periodic geodesic (i.e. there is an isometry $g$ of $X$ such that $g\alpha=\alpha$). 
\end{rem}

We now recall Gersten's definition of divergence as a quasi-isometry invariant from \cite{MR1254309}. We first construct the notions of domination and equivalence which are used in Gersten's definition of divergence.

\begin{defn}Let $\mathcal{M}$ be the collection of all functions from $[0,\infty)$ to $[0,\infty]$. Let $\{\delta_{\rho}\}$ and $\{\delta'_{\rho}\}$ be two families of functions of $\mathcal{M}$, indexed over $\rho \in (0,1]$. \emph{The family $\{\delta_{\rho}\}$ is dominated by the family $\{\delta'_{\rho}\}$}, denoted \emph{$\{\delta_{\rho}\}\preceq \{\delta'_{\rho}\}$}, if there exists a constant $L\in (0,1]$ such that $\delta_{L\rho}\preceq \delta'_{\rho}$. Two families \emph{$\{\delta_{\rho}\}$ and $\{\delta'_{\rho}\}$} are \emph{equivalent}, denoted \emph{$\{\delta_{\rho}\}\sim \{\delta'_{\rho}\}$}, if $\{\delta_{\rho}\}\preceq \{\delta'_{\rho}\}$ and $\{\delta'_{\rho}\}\preceq \{\delta_{\rho}\}$. \emph{The family $\{\delta_{\rho}\}$ is strictly dominated by the family $\{\delta'_{\rho}\}$}, denoted \emph{$\{\delta_{\rho}\}\prec \{\delta'_{\rho}\}$}, if $\{\delta_{\rho}\}$ is dominated by $\{\delta'_{\rho}\}$ and they are not equivalent.
\end{defn}

\begin{rem}
The relations $\preceq$ and $\prec$ are transitive. The relation $\sim$ is an equivalence relation.

If $f$ is an element in $\mathcal{M}$, we could represent $f$ as a family $\{\delta_{\rho}\}$ for which $\delta_{\rho}=f$ for all $\rho$. Therefore, the family $\{\delta_{\rho}\}$ is dominated by (or dominates) a function $f$ in $\mathcal{M}$ if $\{\delta_{\rho}\}$ is dominated by (or dominates) the family $\{\delta'_{\rho}\}$ where $\delta'_{\rho}=f$ for all $\rho$. The equivalence between a family $\{\delta^n_{\rho}\}$ and a function $f$ in $\mathcal{M}$ can be defined similarly. Thus, a family $\{\delta_{\rho}\}$ is linear, quadratic, exponential, etc if $\{\delta_{\rho}\}$ is equivalent to the function $f$ where $f$ is linear, quadratic, exponential, etc.
\end{rem}
 Let $X$ be a geodesic space and $x_0$ one point in $X$. Let $d_{r,x_0}$ be the induced length metric on the complement of the open ball with radius $r$ about $x_0$. If the point $x_0$ is clear from context, we can use the notation $d_r$ instead of using $d_{r,x_0}$.

\begin{defn}
Let $X$ be a geodesic space and $x_0$ one point in $X$. For each $\rho \in (0,1]$, we define a function $\delta_{\rho}\!:[0, \infty)\to [0, \infty)$ as follows: 

For each $r$, let $\delta_{\rho}(r)=\sup d_{\rho r}(x_1,x_2)$ where the supremum is taken over all $x_1, x_2 \in S_r(x_0)$ such that $d_{\rho r}(x_1, x_2)<\infty$.

The family of functions $\{\delta_{\rho}\}$ is the \emph{divergence} of $X$ with respect to the point $x_0$, denoted $Div_{X,x_0}$.

In \cite{MR1254309}, Gersten show that the divergence $Div_{X,x_0}$ is, up to the relation ~$\sim$, a quasi-isometry invariant which is independent of the chosen basepoint. The \emph{divergence} of $X$, denoted $Div_X$, is then, up to the relation ~$\sim$, the divergence $Div_{X,x_0}$ for some point $x_0$ in $X$. 

The \emph{divergence} of a finitely generated group is the divergence of its Cayley graphs.
\end{defn}

\begin{rem}
\label{rma1}

Let $X$ be a one-ended geodesic space with the geodesic extension property. Let $\{\delta_{\rho}\}$ be the divergence of $X$ with respect to some basepoint $x_0$. It is not hard to see that $\delta_{\rho} \sim \delta_1$ for each $\rho \in (0,1]$. In this case, we think of $Div_{X,x_0}$ as the function of $\delta_1$. In particular, we can consider the divergence $Div_X$ of the space $X$, up to the relation ~$\sim$, as a function. Let $\alpha$ and $\beta$ be two rays with the same initial point. It is not hard to see that $\Div_{\alpha,\beta}\preceq Div_X$. In particular, the divergence $Div_X$ is a upper bound of the divergence spectrum of $X$.
\end{rem}

We give some examples of divergence spectra of some spaces as well as finitely generated groups.
\begin{exmp}
 The divergence spectrum of an abelian group is empty since there is no Morse bi-infinite geodesic in its Cayley graphs.

The divergence spectrum of a one-ended hyperbolic group only contains exponential functions since each bi-infinite geodesic in its Cayley graphs is Morse and has exponential lower divergence. 

Charney-Sultan \cite{MR3339446} prove that the lower divergence of a Morse geodesic of a $\CAT(0)$ space has at least quadratic lower divergence. Therefore, the divergence spectrum of a $\CAT(0)$ space is either empty or only contains functions that are at least quadratic.

By the work of Behrstock-Charney \cite{MR2874959}, the divergence spectrum of a one-ended right-angled Artin group is either empty or only contains quadratic functions. Moreover, the divergence spectrum of a one-ended right-angled Artin group is empty iff the group can be written as the direct product of two nontrivial subgroups. 

In \cite{MR2584607}, Dru{\c{t}}u-Mozes-Sapir prove that the existence of a Morse geodesic in a geodesic space $X$ implies the existence of cut points in all asymptotic cones of $X$. Therefore, if the space $X$ has non-empty divergence spectrum, then all asymptotic cones of $X$ contain cut points.

In general, the divergence spectrum of a subgroup does not need to be contained in the divergence spectrum of the whole group. For example, let $H$ be a finitely generated group with non-empty divergence spectrum and $G=H \times K$, where $K$ be any finitely generated group. Then the divergence spectrum of $G$ is empty while the divergence spectrum of $H$ is not. 
\end{exmp}

The divergence spectrum of a one-ended geodesic space $X$ with the geodesic extension property only contains functions that are dominated by $Div_{X}$. Therefore, if $X$, $Y$ are two one-ended geodesic spaces with the geodesic extension property such that $Div_X \prec Div_Y$ and $Div_Y \in S_Y$, then $X$ and $Y$ have different divergence spectra.

\begin{ques}
Let $X$ be a one-ended geodesic space with the geodesic extension property. Does $Div_X$ belong to the divergence spectrum of $X$?
\end{ques}

\section{Divergence spectra of relatively hyperbolic groups}

In this section, we investigate divergence spectra of finitely generated relatively hyperbolic groups.
\label{sporh}

We first recall the concepts of coned off Cayley graphs and relatively hyperbolic groups.
\begin{defn}
Given a finitely generated group $G$ with Cayley graph $\Gamma(G,S)$ equipped with the path metric and a finite collection $\PP$ of subgroups of G, one can construct the \emph{coned off Cayley graph} $\hat{\Gamma}(G,S,\PP)$ as follows: For each left coset $gP$ where $P\in \PP$, add a vertex $v_{gP}$, called a \emph{peripheral vertex}, to the Cayley graph $\Gamma(G,S)$ and for each element $x$ of $gP$, add an edge $e(x,gP)$ of length 1/2 from $x$ to the vertex $v_{gP}$. This results in a metric space that may not be proper (i.e. closed balls need not be compact).
\end{defn}


\begin{defn} [Relatively hyperbolic group]
\label{rel}
A finitely generated group $G$ is \emph{hyperbolic relative to a finite collection $\PP$ of subgroups of $G$} if the coned off Cayley graph is $\delta$--hyperbolic and \emph{fine} (i.e. for each positive number $n$, each edge of the coned off Cayley graph is contained in only finitely many circuits of length $n$).

Each group $P\in \PP$ is a \emph{peripheral subgroup} and its left cosets are \emph{peripheral left cosets} and we denote the collection of all peripheral left cosets by $\Pi$.

An element $g$ of $G$ is \emph{hyperbolic} if $g$ is not conjugate to any element of any peripheral subgroups. 
\end{defn}

We now come up with some concepts and lemmas that help us establish some properties of divergence spectra of relatively hyperbolic groups.
\begin{lem}\cite[Lemma 4.15]{MR2153979}
\label{lemma2}
\label{thm:PeripheralQuasiconvex}
Let $(G,\PP)$ be a relatively hyperbolic group with a finite generating set $S$.
For each $K\geq 1$, $L>0$ and $A_0>0$ there is a constant $A_1=A_1(K,L,A_0)$ such that the following
holds in $\Cayley(G,\Set{S})$.
Let $c$ be a $(K,L)$--quasi-geodesic segment whose endpoints lie in the $A_0$--neighborhood
of a peripheral left coset $gP$.
Then $c$ lies in the $A_1$--neighborhood of $gP$. Thus, if $P$ is a peripheral subgroup of $G$ such that the set $T=S\cap P$ generates $P$, then the Cayley Graph $\Gamma(P,T)$ is a subgraph of $\Gamma(G,S)$ and the embedding map is a quasi-isometric embedding.
\end{lem}

\begin{defn}
Let $c$ be a geodesic of $\Gamma(G,S)$,
and let $\epsilon,R$ be positive constants.
A point $x \in c$ is \emph{$(\epsilon,R)$--deep}
in a peripheral left coset $gP$ (with respect to $c$)
if $x$ is not within a distance $R$ of an endpoint of $c$ and
$\ball{x}{R} \cap c$ lies in $\nbd{gP}{\epsilon}$. A point $x \in c$ is \emph{$(\epsilon,R)$--deep} if $x$ is $(\epsilon,R)$--deep in some peripheral left coset $gP$. If $x$ is not $(\epsilon,R)$--deep in any peripheral left coset $gP$
then $x$ is an \emph{$(\epsilon,R)$--transition point} of $c$.
\end{defn}

\begin{lem} \cite[Lemma 8.10]{Hruska10} 
\label{lemma5a}
Let $(G,\PP)$
be relatively hyperbolic with a finite generating set $S$.
For each $\epsilon$ there is a constant $R=R(\epsilon)$
such that the following holds.
Let $c$ be any geodesic of $\Gamma(G,S)$,
and let $\bar{c}$ be a connected component of the set of all
$(\epsilon,R)$--deep points of $c$.
Then there is a peripheral left coset $gP$ such that each $x \in \bar{c}$
is $(\epsilon,R)$--deep in $gP$ and is not $(\epsilon,R)$--deep in any other
peripheral left coset.
\end{lem}

\begin{lem} \cite[Proposition 8.13]{Hruska10}
\label{lemma5b}
Let $(G,\PP)$ be relatively hyperbolic with a finite generating set $S$.
There exist constants $\epsilon$, $R$ and $L$ such that the following holds.
Let $c$ be any geodesic of $\Gamma(G,S)$ with endpoints in $G$,
and let $\hat{c}$ be a geodesic of $\hat{\Gamma}(G,S,\PP)$
with the same endpoints as $c$.
Then in the metric $d_{S}$, the set of $G$--vertices of $\hat{c}$
is at a Hausdorff distance at most $L$ from the set of
$(\epsilon,R)$--transition points of $c$.
Furthermore, the constants $\epsilon$ and $R$ satisfy the conclusion
of Lemma~ \ref{lemma5a}.
\end{lem}

\begin{lem}\cite[Lemma A.3]{MR2153979}
\label{mainlemma}
Let $(G,\PP)$ be a relatively hyperbolic group with a finite generating set $S$. Then there is a constant $K>1$ such that the following holds. Let $p$ and $q$ be paths in 
$\hat{\Gamma}(G,S,\PP)$ such that $p^{-} =q^{-}$, $p^{+} =q^{+}$, and $q$ is geodesic in $\hat{\Gamma}(G,S,\PP)$. Then for any vertex $v \in q$, there exists a vertex $w \in p$ such that $d_S(w,v)\leq K \log_2 \abs{p}$.
\end{lem}

\begin{rem}
Let $G$ be a finitely generated group with a finite generating set $S$ such that $\Gamma(G,S)$ is one-ended with the geodesic extension property. By Remark \ref{rma1}, we can consider $Div_{\Gamma(G,S),e}$ as a function and it is not hard to see that $Div_{\Gamma(G,S),g}=Div_{\Gamma(G,S),e}$ for any $G$--vertex $g$.
\end{rem}

The following theorem shows that there is a ``gap'' in the divergence spectrum of a relatively hyperbolic group in terms of the divergence of the whole group and the divergence of its peripheral subgroups 
\begin{thm}
\label{tm2}
Let $(G,\PP)$ be a finitely generated relatively hyperbolic group. Suppose that there is a finite generating set $S$ such that for each subgroup $P$ in $\PP$ the set $S\cap P$ generates $P$ and the Cayley graph $\Gamma(P,S\cap P)$ of group $P$ with respect to the generating set $S\cap P$ is one-ended with the geodesic extension property. Let $f=\max\{Div_{\Gamma(P,S\cap P),e}\mid P\in\PP\}$. Then, the divergence spectrum of $G$ only contains functions that are at least exponential or dominated by $f$.
\end{thm}

\begin{proof}
Let $\epsilon$ and $R$ be the constants in Lemma \ref{lemma5b}. Let $\alpha$ be an arbitrary bi-infinite geodesic in $\Gamma(G,S)$. We will show that the lower divergence of $\alpha$ is at least exponential if the length of $(\epsilon,R)$--deep components of $\alpha$ is uniformly bounded and the lower divergence of $\alpha$ is dominated by $f$ otherwise. The following two lemmas complete the proof of this theorem.
\end{proof}
\begin{lem}
Assume that the length of $(\epsilon,R)$--deep components of $\alpha$ is not uniformly bounded. Then the lower divergence of $\alpha$ is dominated by $f$.
\end{lem}
\begin{proof}
There are a constants $L\geq 1$ and $K>0$ such that the natural embedding of $\Gamma(P,P\cap S)$ into $\Gamma(G,S)$ is $(L,K)$--quasi-isometry. Let $M=32L^2$ and $D=8\epsilon+8K$. We are going to prove that for each $r>D$ \[ldiv_{\alpha}(r)\leq f\bigl(\frac{Mr}{8L}\bigr)+2Mr.\] 
Since the length of $(\epsilon,R)$--deep components of $\alpha$ is not uniformly bounded, then there is an $(\epsilon,R)$--deep component $c$ of $\alpha$ with length at least $Mr$. Assume that $c$ is an $(\epsilon,R)$--deep component with respect to the peripheral left coset $gP$. Let $Y$ be the translation of $\Gamma(P,P\cap S)$ by $g$. Thus, the natural embedding of $Y$ in $\Gamma(G,S)$ is also an $(L,K)$--quasi-isometry and $gP$ is the set of vertices of the graph $Y$. 

Let $x_1$, $x_2$ be two points in $c$ such that $d_S(x_1,x_2)=Mr$. Let $x$ be a midpoint of the subsegment of $c$ with endpoints $x_1$ and $x_2$ such that $d_S(x,x_1)=d_S(x,x_2)=Mr/2$. Let $y_1$, $y_2$ and $y$ be elements in $gP$ such that $d_S(x_1,y_1)<\epsilon$, $d_S(x_2,y_2)<\epsilon$ and $d_S(x,y)<\epsilon$. Therefore, \[d_S(y_1,y)\geq d_S(x_1,x)-d_S(x_1,y_1)-d_S(x,y)\geq \frac{Mr}{2}-2\epsilon\geq\frac{Mr}{4}.\]
Thus, \[d_Y(y_1,y)\geq \frac{1}{L}d_S(y_1,y)-K\geq \frac{Mr}{4L}-K\geq \frac{Mr}{8L}.\]
Therefore, there is a point $y_3$ in the sphere with radius $Mr/{8L}$ about $y$ in $Y$ and a path $\gamma_1$ from $y_1$ to $y_3$ which lies outside the ball with radius $Mr/{8L}$ about $y$ in $Y$. Moreover, the length of $\gamma_1$ is bounded above by $Mr/{8L}$. Similarly, there is a point $y_4$ in the sphere with radius $Mr/{8L}$ about $y$ in $Y$ and a path $\gamma_2$ from $y_2$ to $y_4$ which lies outside the ball with radius $Mr/{8L}$ about $y$ in $Y$. Moreover, the length of $\gamma_2$ is bounded above by $Mr/{8L}$. Let $\gamma_3$ be a path from $y_3$ to $y_4$ which lies outside the ball with radius $Mr/{8L}$ about $y$ in $Y$ and the length of $\gamma_3$ is bounded above by $Div_{Y,g}(Mr/{8L})$. Obviously, $Div_{Y,g}=Div_{\Gamma(P,P\cap S),e}$ is bounded by the function $f$. Therefore, the length of $\gamma_3$ is bounded above by $f(Mr/{8L})$. Let $\gamma_4=\gamma_1\cup\gamma_3\cup\gamma_2$. Then the length of $\gamma_4$ is bounded above by $f(Mr/{8L})+Mr/{4L}$. We need to show that $\gamma_4$ lies outside the open ball with radius $2r$ about $x$ in $\Gamma(G,S)$.

Indeed, for each $u$ in $\gamma_4$, \[d_S(u,x)\geq d_S(y,u)-d_S(y,x)\geq \frac{1}{L}d_Y(y,u)-\frac{K}{L}-\epsilon\geq \frac{Mr}{8L^2}-\frac{K}{L}-\epsilon\geq \frac{Mr}{16L^2}\geq 2r.\]

Since $d_S(x_1,y_1)<\epsilon$, then there is a path $\gamma_5$ from $x_1$ to $y_1$ with length bounded by $\epsilon$. Similarly, there is a path $\gamma_6$ from $x_2$ to $y_2$ with length bounded by $\epsilon$. Let $\gamma_7=\gamma_5\cup\gamma_4\cup \gamma_6$. Then $\gamma_7$ is a path from $x_1$ to $x_2$ which lies outside the open ball with radius $r$ about $x$ in $\Gamma(G,S)$. Moreover, the length of $\gamma_7$ is bounded above by $f(Mr/{8L})+Mr/{4L}+2\epsilon$. Thus, the length of $\gamma_7$ is bounded above by $f(Mr/{8L})+Mr/{2L}$ by the choice of $r$.

Let $x_3$ be the point in $c$ which lies between $x_1$ and $x$ such that $d_S(x,x_3)=r$. Let $\gamma_8$ be a subsegment of $c$ with endpoints $x_1$, $x_3$. Then, $\gamma_8$ lies outsides the open ball with radius $r$ about $x$. Since $d_S(x,x_1)=Mr/2$, then the length of $\gamma_8$ is bounded above by $Mr/2$. Let $x_4$ be the point in $c$ which lies between $x_2$ and $x$ such that $d_S(x,x_4)=r$. Let $\gamma_9$ be a subsegment of $c$ with endpoints $x_2$, $x_4$. Then, $\gamma_9$ lies outsides the open ball with radius $r$ about $x$. Since $d_S(x,x_2)=Mr/2$, then the length of $\gamma_9$ is also bounded above by $Mr/2$. Let $\gamma_{10}=\gamma_8\cup\gamma_7\cup\gamma_9$. Then $\gamma_{10}$ is a path from $x_3$ to $x_4$ which lies outside the open ball with radius $r$ about $x$. Moreover, the length of $\gamma_{10}$ is bounded above by $f(Mr/{8L})+Mr/{2L}+Mr$. Thus, the length of $\gamma_{10}$ is bounded above by $f(Mr/{8L})+2Mr$ by the choice of $r$. Therefore, \[ldiv_{\alpha}(r)\leq f(\frac{Mr}{8L})+2Mr,\]
 which implies that $ldiv_{\alpha} \preceq f$.
\end{proof}

\begin{lem}
Assume that the length of $(\epsilon,R)$--deep components of $\alpha$ is uniformly bounded. Then the lower divergence of $\alpha$ is at least exponential.
\end{lem}
\begin{proof}
Suppose that the length of each $(\epsilon,R)$--deep components of $\alpha$ is bounded by a number $D$. Let $K$ be the constant in Lemma \ref{mainlemma} and $L$ the constant in Lemma \ref{lemma5b}. We need to show that $ldiv_{\alpha}(2Kr)\geq 2^r$ for each $r>D+L$.

For each $t\in\RR$, let $\gamma$ be an arbitrary path from $\alpha(t-2Kr)$ to $\alpha(t+2Kr)$ which lies outside the open ball with radius $2Kr$ about $\alpha(t)$. Since the length of each $(\epsilon,R)$--deep components of $\alpha$ is bounded by $D$, then there is an $(\epsilon,R)$--transition point $u$ of $\alpha$ such that the distance between $\alpha(t)$ and $u$ is bounded above by $D$. Thus, $\gamma$ lies outside the open ball with radius $2Kr-D$ about $u$. Let $\bar{\alpha}$ be a geodesic in $\hat{\Gamma}(G,S,\PP)$ connecting $\alpha(t-2Kr)$ and $\alpha(t+2Kr)$. By Lemma \ref{lemma5b}, there is a $G$--vertex $v$ in $\bar{\alpha}$ such that the distance between $u$ and $v$ is bounded above by $L$ with respect to $d_S$. Thus, $\gamma$ lies outside the open ball with radius $2Kr-D-L$ about $v$. Therefore, $\gamma$ lies outside the open ball with radius $Kr$ about $v$ by the choice of $r$.

By Lemma \ref{mainlemma}, there is a point $w$ in $\gamma$ such that $d_S(v,w)\leq K \log_2{\abs{\gamma}}$. Also, the distance between $v$ and $w$ is bounded below by $Kr$. Then, the length of $\gamma$ is bounded below by $2^r$. Thus, $\rho_{\alpha}(2Kr,t)\geq 2^r$. Since $t$ is an arbitrary number and $ldiv_{\alpha}(2Kr)=\inf_{t\in\RR}\rho_{\alpha}(2Kr,t)$, then $ldiv_{\alpha}(2Kr) \geq 2^r$.
Therefore, the lower divergence of $\alpha$ is at least exponential.
\end{proof}

We now show the connection between the divergence spectrum of a relatively hyperbolic group and the divergence spectra of its peripheral subgroups. We first prove some results on relatively hyperbolic groups that help us establish the connection.

\begin{lem}
\label{b4}
Let $(G,\PP)$ be a relatively hyperbolic group with finite generating set $S$. Let $P$ be a peripheral subgroup of $G$ such that the set $T=S\cap P$ generates $P$. Let $L\geq 1$ be a number such that $d_T(u,v)\leq L~d_S(u,v)+L$ for each $u$, $v$ in $\Gamma(P,T)$. Let $r$ be an arbitrary number greater than $L$. Let $\gamma$ a simple path in $\Gamma(G,S)$ such that $\gamma$ lies in the $r$--neighborhood of $P$ and the endpoints of $\gamma$ lie in $\Gamma(P,T)$. If the length of $\gamma$ is greater than $r$, then there is a path $\gamma'$ in $\Gamma(P,T)$ with the same endpoints as $\gamma$ such that the Hausdorff distance between $\gamma$, $\gamma'$ is bounded above by $6Lr$ and $\abs{\gamma'}\leq 4L~\abs{\gamma}$. 
\end{lem}

\begin{proof}
We divide $\gamma$ into $m$ subpaths $\gamma_1, \gamma_2,\cdots, \gamma_m$ with length between $r$ and $2r$ by $m+1$ points $u_0, u_1, \cdots, u_m$, where $u_0$ and $u_m$ are the two endpoints of $\gamma$. Then, $\abs{\gamma}=\abs{\gamma_1}+\abs{\gamma_2}+\cdots+\abs{\gamma_m}$. Let $v_0, v_1, \cdots, v_m$ be elements in $\PP$ such that $d_S(u_i,v_i)<r$, $u_0=v_0$ and $u_m=v_m$. Let $\gamma'_i$ be a geodesic in $\Gamma(P,T)$ connecting $v_i$, $v_{i-1}$ and $\gamma'=\gamma'_1 \cup \gamma'_2\cup\cdots\cup\gamma'_m$. For each $i$, 
\[d_T(v_i,v_{i-1})\leq L~d_S(v_i,v_{i-1})+L\leq L\bigl(d_S(v_i,u_i)+d_S(u_i,u_{i-1})+d_S(u_{i-1}, v_{i-1})\bigr)+L\leq L~\bigl(\abs{\gamma_i}+2r\bigr)+L\]
Thus, $\abs{\gamma'}\leq L\bigl(\abs{\gamma}+2rm\bigr)+Lm\leq L\bigl(\abs{\gamma}+3rm\bigr)\leq 4L~\abs{\gamma}$

For each $u$ in $\gamma$, $u$ must lie in some $\gamma_i$. Thus, the distance between $u$ and $u_i$ is bounded above by $2r$. Also, the distance between $u_i$ and $v_i$ is bounded above by $r$. Then, the distance between $u$ and $v_i$ is bounded above by $3r$. Thus, $\gamma$ lies in the $3r$--neighborhood of $\gamma'$. For each $v$ in $\gamma'$, $v$ must lie in some $\gamma'_i$. Thus, the distance between $v$ and $v_i$ is bounded above by the length of $\gamma'_i$. Also, length of $\gamma'_i$ is bounded above by $L~\bigl(\abs{\gamma_i}+2r\bigr)+L$ and length of $\gamma_i$ is bounded above by $2r$. Then, the distance between $v$ and $v_i$ is bounded above by $4Lr+L$. Since the distance between $v_i$ and $u_i$ is bounded above by $r$, then the distance between $v$ and $u_i$ is bounded above by $4Lr+L+r$. Thus, $\gamma'$ lies in the $6Lr$--neighborhood of $\gamma$. Therefore, the Hausdorff distance between $\gamma$, $\gamma'$ is bounded above by $6Lr$
\end{proof}

\begin{lem}\cite[Lemma 8.27] {MR3361149}
\label{lemman1}
Let $(G,\PP)$ be relatively hyperbolic with a finite generating set $S$.
There exist constants $\epsilon$, $R$, $\sigma$, $K$ and $A$ such that the following hold: 
\begin{enumerate}
\item Let $p$ and $q$ be paths in $\Gamma(G,S)$ such that $p^{-} =q^{-}$, $p^{+} =q^{+}$ and $q$ is geodesic in $\Gamma(G,S)$. For any $(\epsilon, R)$--transition point $v \in q$, there exists a vertex $w \in p$ such that $d_S(w,v)\leq K \log_2 \abs{p}+K$.
\item For each peripheral left coset $gP$ and any geodesic $c$ with endpoints outside $N_A(gP)$. If $\ell(c)>9\max\bigl\{d_S(c^+,gP);d_S(c^-,gP)\bigr\}$, then the path $c$ contains an $(\epsilon, R)$--transition point $w$ which lies in the $A$--neighborhood of $gP$.
\end{enumerate}
 Furthermore, the constants $\epsilon$ and $R$ satisfy the conclusion of Lemma~ \ref{lemma5a}.
\end{lem}

\begin{prop}
\label{p1}
Let $(G,\PP)$ be a relatively hyperbolic group with finite generating set $S$. Let $P$ be a peripheral subgroup of $G$ such that the set $T=S\cap P$ generates $P$. For each function $M\!: [1,\infty)\times [0,\infty)\rightarrow [0,\infty)$ there is a function $N\!: [1,\infty)\times [0,\infty)\rightarrow [0,\infty)$ such that the following hold. Let $\beta$ be an arbitrary $M$--Morse bi-infinite geodesic (geodesic ray) in $\Gamma(P,T)$. Then there is a $N$--Morse bi-infinite geodesic (geodesic ray with the same initial point as $\beta$) $\alpha$ in $\Gamma(G,S)$ such that the Hausdorff distance between $\alpha$ and $\beta$ is finite.
\end{prop}

\begin{proof}
Here we only give the proof for the case that $\beta$ is a bi-infinite geodesic. The proof of the case that $\beta$ is a geodesic ray is almost identical. 

There are constants $L\geq 1$, $D>0$ such that $d_S(u,v)\leq d_T(u,v)\leq Ld_S(u,v)+L$ for each $u$, $v$ in $\Gamma(P,T)$ and every geodesic $\sigma$ in $\Gamma(G,S)$ with endpoints in $P$ must lie in the $D$--neighborhood of $P$ by Lemma \ref{lemma2}. For each $n$, let $\alpha_n$ be a geodesic in $\Gamma(G,S)$ connecting $\beta(-n)$ and $\beta(n)$. We divide a geodesic $\alpha_n$ in to $m$ segments by $(m+1)$ points $u_0, u_1,\cdots u_m$ in $\alpha_n$ such that $1/2\leq d_S(u_i,u_{i-1})\leq 1$. We choose $u_0=\beta(-n)$ and $u_m=\beta(n)$. Let $v_0, v_1,\cdots v_m$ be points in $P$ such that $d_S(u_i,v_i)<D$. We choose $v_0=\beta(-n)$ and $v_m=\beta(n)$. We connect $v_i$ and $v_{i-1}$ by a geodesic in $\Gamma(P,T)$ to construct a path $\gamma_n$ from $\beta(-n)$ to $\beta(n)$. It is not hard to see that $\gamma_n$ is a $(K_1, L_1)$--quasi-geodesic in $\Gamma(P,T)$ and the Hausdorff distance between $\gamma_n$, $\beta_n$ is bounded above by $D_1$, where $K_1, L_1, D_1$ only depends on $D$ and $L$. 

Since $\beta$ is a $M$--Morse geodesic in $\Gamma(P,T)$, then there is $D_2=D_2(M, K_1,L_1)$ such that the Hausdorff distance between $\gamma_n$ and $\beta\bigl([-n,n]\bigr)$ is bounded by the constant $D_2$ with respect to the metric $d_T$ for each $n$. Obviously, $D_2$ only depends on $D$, $L$, and $M$. Also $d_S(u,v)\leq d_T(u,v)$ for each $u$, $v$ in $\Gamma(P,T)$. Then the Hausdorff distance between $\gamma_n$ and $\beta\bigl([-n,n]\bigr)$ is also bounded by some constant $D_2$ with respect to the metric $d_S$ for each $n$. Therefore, the Hausdorff distance between $\beta\bigl([-n,n]\bigr)$ and $\alpha_n$ is bounded above by $D_1+D_2$. Since $\Gamma(G,S)$ is a proper space, then there is a Morse bi-infinite geodesic $\alpha$ in $\Gamma(G,S)$ such that the Hausdorff distance between $\alpha$ and $\beta$ is bounded above by $D_3=D_1+D_2+1$. Again, $D_3$ only depends on $D$, $L$, and $M$. 

We now prove that $\alpha$ is a $N$--Morse geodesic in $\Gamma(G,S)$, where $N$ only depends on the geometry of $\Gamma(G,S)$ and $M$. For each $K_0\geq 1$, $L_0>0$, there is $C=C(K_0, L_0, D_3)$ such that each $(K_0,L_0)$--quasi-geodesic segment whose endpoints lie in the $D_3$--neighborhood of $P$ must lie in the $C$--neighborhood of $P$ (see Lemma \ref{lemma2}). We can assume that $C>D_3$. Let $\sigma\!:[s_1, s_2]\to \Gamma(G,S)$ be any $(K_0,L_0)$--quasi-geodesic in $\Gamma(G,S)$ with endpoints on $\alpha$. Then, the endpoints of $\sigma$ lie in the $D_3$--neighborhood of $\beta$. In particular, the endpoints of $\sigma$ lie in the $D_3$--neighborhood of $P$. Thus, $\sigma$ lies in the $C$--neighborhood of $P$. We subdivide $[s_1,s_2]$ into m subintervals with lengths less than 1 and greater than 1/2 by $(m+1)$ numbers as follows:
\[s_1=w_0<w_1<\cdots<w_m=s_2.\]
Let $a_i$ be an element of $P$ such that $d_S(\sigma(w_i), a_i)\leq C$. We choose $a_0$ and $a_m$ in $\beta$. We connect $a_i$ and $a_{i-1}$ by a geodesic in $\Gamma(P,T)$ to construct a path $\sigma'$ with endpoints on $\beta$. It is not hard to see that $\sigma'$ is a $(K', L')$--quasi-geodesic in $\Gamma(P,T)$ and the Hausdorff distance between $\sigma$, $\sigma'$ is bounded above by $D'$, where $K'$, $L'$, and $D'$ only depend on and $D$, $L$, $M$, $K_0$, $L_0$ and $C$. There is a constant $M_1=M(K',L')$ such that $\sigma'$ lies in the $M_1$--neighborhood of $\beta$ with respect to $d_T$. Thus, $\sigma'$ also lies in the $M_1$--neighborhood of $\beta$ with respect to $d_S$. Therefore, $\sigma$ lies in the $(M_1+D')$--neighborhood of $\beta$ with respect to $d_S$. Also, the Hausdorff distance between $\alpha$ and $\beta$ is bounded above by $D_3$. Then, $\sigma$ lies in the $(M_1+D'+D_3)$--neighborhood of $\alpha$ with respect to $d_S$. Obviously, $M_1$, $D'$, and $D_3$ originally depends on $D$, $L$, $M$, $K_0$, $L_0$ and $C$. Therefore, $\alpha$ is a $N$--Morse geodesic in $\Gamma(G,S)$, where $N$ only depends on the geometry of $\Gamma(G,S)$ and $M$. 
 \end{proof}

\begin{lem}
\label{b3}
Let $(G,\PP)$ be relatively hyperbolic with a finite generating set $S$. There exist constants $K$ and $B$ such that the following hold. Let $gP$ be an arbitrary peripheral left coset. Let $r$ be a number greater than $B$ and $\gamma$ an arbitrary path that lies outside $N_r(gP)$ such that the endpoints of $\gamma$ lie in $\partial N_r(gP)$. If the distance between two endpoints of $\gamma$ is greater than $10r$, then the length of $\gamma$ is bounded below by $2^{r/{2K}}$.
\end{lem}

\begin{proof}
Let $\epsilon$, $R$, $\sigma$, $K$, $A$ be constants in Lemma \ref{lemman1} and let $B=2(A+K)$. Let $c$ be a geodesic in $\Gamma(G,S)$ with the same endpoints with $\gamma$. By Lemma \ref{lemman1}(2), $c$ contains an $(\epsilon, R)$--transition point $w$ which lies in the $A$--neighborhood of $gP$. By Lemma \ref{lemman1}(1), there exists a vertex $v \in \gamma$ such that $d_S(w,v)\leq K \log_2 \abs{\gamma}+K$. Also, 
\[d_S(w,v)\geq d_S(v,gP)-d_S(w,gP)\geq r-A\geq \frac{r}{2}+K.\]
Thus, $K \log_2 \abs{\gamma}\leq r/2$. Then, the length of $\gamma$ is bounded below by $2^{r/{2K}}$.
 \end{proof}

The function $f$ is \emph{subexponential} if for each $a>1$, there is $C>0$ such that $f(x)<a^x$ for all $x>C$.

\begin{prop}
\label{p2}
Let $(G,\PP)$ be a relatively hyperbolic group with finite generating set $S$. Let $P$ be a peripheral subgroup of $G$ such that the set $T=S\cap P$ generates $P$. Let $\alpha$ be a bi-infinite geodesic in $\Gamma(G,S)$ and $\beta$ bi-infinite geodesic in $\Gamma(P,T)$ such that the Hausdorff distance between them is finite. If the lower divergence of $\beta$ in $\Gamma(P,T)$ is subexponential, then the lower divergence of $\beta$ in $\Gamma(P,T)$ is equivalent to the lower divergence of $\alpha$ in $\Gamma(G,S)$.
\end{prop}

\begin{proof}
Let $f$ be the lower divergence of $\beta$ in $\Gamma(P,T)$ and $g$ the lower divergence of $\alpha$ in $\Gamma(G,S)$. There is a constant $L\geq 1$ such that $d_T(u,v)\leq L~d_S(u,v)+L$ for each $u$, $v$ in $\Gamma(P,T)$. Let $L_1$ be the Hausdorff distance between $\alpha$ and $\beta$. The following two lemmas complete the proof of this proposition. 
\end{proof}

\begin{lem}
The function $f$ is dominated by the function $g$
\end{lem}

\begin{proof}
Let $B$, $K$ be constants in Lemma \ref{b3} and $C$ a constant such that $f(r) < 2^{r/{2K}}-2L_1$ for each $r>C$. Let $M=84L$ and we need to show that for each $r>B+C+4L_1+2L+2$, the following inequality hold: \[f(r)\leq 4L~g(Mr)+2L(M+3)r.\]

For any number $t\in\RR$, let $\gamma$ be an arbitrary path from $\alpha(t-Mr)$ to $\alpha(t+Mr)$ that lies outside the open ball with radius $Mr$ about $\alpha(t)$ with respect to the metric $d_S$. We are going to show \[f(r)\leq 4L~\abs{\gamma}+2L(M+3)r.\] We can assume that $\gamma$ is a simple path. Let $s_1$ and $s_2$ be two real numbers such that $d_S\bigl(\beta(s_1), \alpha(t-Mr)\bigr)\leq L_1$ and $d_S\bigl(\beta(s_2), \alpha(t+Mr)\bigr)\leq L_1$. Let $\gamma_1$ be a geodesic connecting $\alpha(t-Mr)$ and $\beta(s_1)$. Let $\gamma_2$ be a geodesic connecting $\alpha(t+Mr)$ and $\beta(s_2)$. Let $\bar{\gamma}=\gamma_1\cup \gamma \cup \gamma_2$ then $\bar{\gamma}$ is a path from $\beta(s_1)$ to $\beta(s_2)$ which lies outside the open ball with radius $(Mr-L_1)$ about $\alpha(t)$. We can assume that $\bar{\gamma}$ is a single path.

Let $J$ be a subinterval of $\RR$ with endpoints $s_1$ and $s_2$. We need to show that there is $s\in J$ such that $d_S\bigl(\alpha(t), \beta(s)\bigr)\leq 3L_1+1$.

We subdivide $J$ into m subintervals with lengths less than 1 by $(m+1)$ numbers as follows:
\[s_1=w_0<w_1<\cdots<w_m=s_2 \text{ if } s_1<s_2\] 
or \[s_2=w_0<w_1<\cdots<w_m=s_1 \text{ if } s_2<s_1.\] 
Since the Hausdorff distance between $\alpha$ and $\beta$ is less than $L_1$, then there is $w'_i$ in $\RR$ such that $d_S\bigl(\alpha(w'_i), \beta(w_i)\bigr)\leq L_1$. We choose $w'_0=t-Mr$, $w'_m=t+Mr$ if $s_1<s_2$ and $w'_0=t+Mr$, $w'_m=t-Mr$ if $s_2<s_1$. Let $J_i$ be a subinterval of $\RR$ with end points $w'_{i-1}$ and $w'_i$. Obviously, $[t-Mr, t+Mr] \subset \bigcup J_i$. Thus, $t$ lies in some $J_i$. Therefore, \begin{align*} d_S\bigl(\alpha(t), \beta(w_i)\bigr)& \leq d_S\bigl(\alpha(t), \alpha(w'_i)\bigr)+d_S\bigl(\alpha(w'_i), \beta(w_i)\bigr) \\&\leq \abs{t-w'_i}+L_1 \\&\leq \abs{w'_i-w'_{i-1}}+L_1.\end{align*}

Also, \begin{align*} \abs{w'_i-w'_{i-1}}&= d_S\bigl(\alpha(w'_i), \alpha(w'_{i-1})\bigr)\\&\leq d_S\bigl(\alpha(w'_i), \beta(w_i)\bigr)+d_S\bigl(\beta(w_i), \beta(w_{i-1})\bigr)+d_S\bigl(\beta(w_{i-1}), \alpha(w'_{i-1})\bigr)\\&\leq d_S\bigl(\alpha(w'_i), \beta(w_i)\bigr)+d_T\bigl(\beta(w_i), \beta(w_{i-1})\bigr)+d_S\bigl(\beta(w_{i-1}), \alpha(w'_{i-1})\bigr)\\&\leq L_1+\abs{w_i-w_{i-1}}+L_1\leq 2L_1+1\end{align*}
Thus, $d_S\bigl(\alpha(t), \beta(w_i)\bigr)\leq 3L_1+1$. Let $s=w_i$ then $\bar{\gamma}$ is a path from $\beta(s_1)$ to $\beta(s_2)$ which lies outside the open ball with radius $(Mr-4L_1-1)$ about $\beta(s)$. We now show that $r\leq \abs{s-s_1}\leq L(M+1)r$. In fact, \begin{align*} \abs{s-s_1}&= d_T\bigl(\beta(s), \beta(s_1)\bigr)\\&\leq Ld_S\bigl(\beta(s),\beta(s_1)\bigr)+L\\&\leq L\biggl(d_S\bigl(\beta(s), \alpha(t)\bigr)+d_S\bigl(\alpha(t), \alpha(t-Mr)\bigr)+d_S\bigl(\alpha(t-Mr),\beta(s_1)\bigr)\biggr)+L\\&\leq L(3L_1+1+Mr+L_1)+L\\&\leq L(M+1)r 
\end{align*}

and \begin{align*} \abs{s-s_1}&= d_T\bigl(\beta(s), \beta(s_1)\bigr)\\&\geq d_S\bigl(\beta(s),\beta(s_1)\bigr)\\&\geq d_S\bigl(\alpha(t), \alpha(t-Mr)\bigr)-d_S\bigl(\alpha(t),\beta(s)\bigr)-d_S\bigl(\alpha(t-Mr),\beta(s_1)\bigr)\\&\geq (Mr-3L_1-1-L_1)\\&\geq r\end{align*}

Similarly, $r\leq \abs{s-s_2}\leq L(M+1)r$

Let $s_3$ be the number between $s$ and $s_1$ such that $\abs{s-s_3}=r$. Let $s_4$ be the number between $s$ and $s_2$ such that $\abs{s-s_3}=r$. Let $\beta_1$ be a subsegment of $\beta$ connecting $\beta(s_1)$ and $\beta(s_3)$. Let $\beta_2$ be a subsegment of $\beta$ connecting $\beta(s_2)$ to $\beta(s_4)$. Then $\beta_1$ and $\beta_2$ both lie outside the open ball with radius $r$ about $\beta(t)$ with respect metric $d_T$. Moreover, the length of $\beta_1$ and $\beta_2$ are both bounded by $L(M+1)r$. 

First we assume that $\bar{\gamma}$ lies in the $r$--neighborhood of $P$. By Lemma \ref{b4}, there is a path $\rho$ in $\Gamma(P,T)$ with the same endpoints with $\bar{\gamma}$, the Hausdorff distance between $\bar{\gamma}$, $\rho$ is bounded by $6Lr$ and $\abs{\rho}\leq 4L~\abs{\bar{\gamma}}$. Since $\bar{\gamma}$ lies outside the open ball with radius $(Mr-4L_1-1)$ about $\beta(s)$ with respect to the metric $d_S$, then $\rho$ lies outside the open ball with radius $(M-6L)r-4L_1-1$ about $\beta(s)$ with respect to the metric $d_S$. Thus, $\rho$ lies outside the open ball with radius $(M-6L)r-4L_1-1$ about $\beta(s)$ with respect to the metric $d_T$. Since $(M-6L)r-4L_1-1>r$, then $\rho$ lies outside the open ball with radius $r$ about $\beta(s)$ with respect to the metric $d_T$. Let $\gamma'=\beta_1\cup \rho\cup\beta_2$ then $\gamma'$ is a path from $\beta(s_3)$ to $\beta(s_4)$ that lies outside the ball with radius $r$ about $\beta(s)$ with respect to the metric $d_T$. Thus, \begin{align*} f(r)\leq \abs{\gamma'}&\leq \abs{\beta_1}+ \abs{\rho}+\abs{\beta_2}\\&\leq 4L~\abs{\bar{\gamma}}+2L(M+1)r\\&\leq 4L~\bigl(\abs{\gamma}+2L_1\bigr)+2L(M+1)r\\&\leq 4L~\abs{\gamma}+2L(M+3)r.\end{align*} 

We now assume that $\bar{\gamma}$ does not lie in the $r$--neighborhood of $P$. Let $x_1, x_2, \cdots, x_{2n-1}, x_{2n}$ be points in $\bar{\gamma} \cap \partial N_r(P)$ such that the subpath $\bar{\gamma_i}$ of $\bar{\gamma}$ connecting $x_{2i-1}$ to $x_{2i}$ must lie out side $N_r(P)$ and $\bar{\gamma}-(\cup \bar{\gamma_i})$ must lie inside $\bar{N_r(P)}$. If $d_S(x_{2i-1},x_{2i})>10r$ for some $i$, then the length of $\bar{\gamma_i}$ is bounded below by $2^{r/{2K}}$ by Lemma \ref{b3}. Thus, the length of $\bar{\gamma}$ is bounded below by $2^{r/{2K}}$. Therefore, the length of $\gamma$ is bounded below by $2^{r/{2K}}-2L_1$. This implies that \[f(r)\leq \abs{\gamma}\leq 4L~\abs{\gamma}+2L(M+3)r.\] 

We now can assume that $d_S(x_{2i-1},x_{2i})\leq 10r$. In the path $\bar{\gamma}$, we replace each $\bar{\gamma_i}$ by a geodesic with the same endpoints of $\bar{\gamma_i}$. Then, the new path $\rho_1$ must lie inside $N_{12r}(P)$. Obviously, $\rho_1$ lies in the $10r$--neighborhood $\bar{\gamma}$ and the length of $\rho_1$ is bounded above the length of $\bar{\gamma}$. Therefore, $\rho_1$ lies outside the open ball with radius $(M-10)r-4L_1-1$ about $\beta(s)$. By Lemma \ref{b4}, there is a path $\rho_2$ in $\Gamma(P,T)$ with the same endpoints with $\rho_1$, the Hausdorff distance between $\rho_1$, $\rho_2$ is bounded by $72Lr$ and $\abs{\rho_2}\leq 4L~\abs{\rho_1}$. Therefore, $\rho_2$ lies outside the open ball with radius $(M-72L-10)r-4L_1-1$ about $\beta(s)$ with respect to the metric $d_S$. Thus, $\rho_2$ lies outside the open ball with radius $r$ about $\beta(s)$ with respect to the metric $d_S$ by the choice of $M$ and $r$. Obviously, $\beta_2$ lies outside the open ball with radius $r$ about $\beta(s)$ with respect to the metric $d_T$. Let $\rho_3=\beta_1 \cup \rho_2 \cup \beta_2$, then $\rho_3$ is a path from $\beta(s_3)$ to $\beta(s_4)$ that lies outside the open ball with radius $r$ about $\beta(s)$ with respect to the metric $d_T$. Therefore, \begin{align*} f(r)\leq \abs{\rho_3}&\leq 4L~\abs{\rho_1}+2L(M+1)r\\&\leq 4L~\abs{\bar{\gamma}}+2L(M+1)r\\&\leq 4L~\bigl(\abs{\gamma}+2L_1\bigr)+2L(M+1)r\\&\leq 4L~\bigl(\abs{\gamma}\bigr)+2L(M+3)r.\end{align*} Thus, \[f(r)\leq 4L~g(Mr)+2L(M+3)r.\]
Therefore, $f\preceq g$.
\end{proof}

\begin{lem}
The function $g$ is dominated by the function $f$
\end{lem}

\begin{proof}
We will show that $g(r)\leq f(2Lr)+4(L+1)r$ for each $r>2L(2L_1+1)+3L_1+1$. For each $t\in\RR$, let $\gamma$ be an arbitrary path from $\beta(t-2Lr)$ to $\beta(t+2Lr)$ that lies outside the open ball with radius $2Lr$ about $\beta(t)$ with respect to the metric $d_T$. Since $d_T(u,v)\leq L~d_S(u,v)+L$ for each $u$, $v$ in $\Gamma(P,T)$, then $\gamma$ lies outside the open ball with radius $2r-1$ about $\beta(t)$ with respect to the metric $d_S$. Let $s_1$ and $s_2$ in $\RR$ such that $d_S\bigl(\alpha(s_1), \beta(t-2Lr)\bigr)\leq L_1$ and $d_S\bigl(\alpha(s_2), \beta(t+2Lr)\bigr)\leq L_1$. Let $\gamma_1$ be a geodesic connecting $\beta(t-2Lr)$, $\alpha(s_1)$ and $\gamma_2$ be a geodesic connecting $\beta(t+2Lr)$, $\alpha(s_2)$. Let $\bar{\gamma}=\gamma_1\cup \gamma \cup \gamma_2$ then $\bar{\gamma}$ is a path from $\alpha(s_1)$ to $\alpha(s_2)$ which lies outside the open ball with radius $(2r-L_1-1)$ about $\beta(t)$.

Let $J$ be a subinterval of $\RR$ with endpoints $s_1$ and $s_2$. We need to show that there is $s\in J$ such that $d_S\bigl(\beta(t), \alpha(s)\bigr)\leq 2L(L_1+1)+L_1$.

We subdivide $J$ into m subintervals with lengths less than 1 by $(m+1)$ numbers as follows:
\[s_1=w_0<w_1<\cdots<w_m=s_2 \text{ if } s_1<s_2\] 
or \[s_2=w_0<w_1<\cdots<w_m=s_1 \text{ if } s_2<s_1.\] 
Since the Hausdorff distance between $\alpha$ and $\beta$ is less than $L_1$, then there is $w'_i$ in $\RR$ such that $d_S\bigl(\beta(w'_i), \alpha(w_i)\bigr)\leq L_1$. We choose $w'_0=t-2Lr$, $w'_m=t+2Lr$ if $s_1<s_2$ and $w'_0=t+2Lr$, $w'_m=t-2Lr$ if $s_2<s_1$. Let $J_i$ be a subinterval of $\RR$ with end points $w'_{i-1}$ and $w'_i$. Obviously, $[t-2Lr, t+2Lr] \subset \bigcup J_i$. Thus, $t$ lies in some $J_i$. Therefore, \begin{align*} d_S\bigl(\beta(t), \alpha(w_i)\bigr)& \leq d_S\bigl(\beta(t), \beta(w'_i)\bigr)+d_S\bigl(\beta(w'_i), \alpha(w_i)\bigr)\\&\leq d_T\bigl(\beta(t), \beta(w'_i)\bigr)+d_S\bigl(\beta(w'_i), \alpha(w_i)\bigr) \\&\leq \abs{t-w'_i}+L_1 \\&\leq \abs{w'_i-w'_{i-1}}+L_1.\end{align*}

Also, \begin{align*} \abs{w'_i-w'_{i-1}}&= d_T\bigl(\beta(w'_i), \beta(w'_{i-1})\bigr)\\&\leq Ld_S\bigl(\beta(w'_i), \beta(w'_{i-1})\bigr)+L\\& \leq L\bigg(d_S\bigl(\beta(w'_i), \alpha(w_i)\bigr)+d_S\bigl(\alpha(w_i), \alpha(w_{i-1})\bigr)+d_S\bigl(\alpha(w_{i-1}), \beta(w'_{i-1})\bigr)\bigg)+L\\&\leq L(L_1+\abs{w_i-w_{i-1}}+L_1)+L\leq 2L(L_1+1)\end{align*}
Thus, $d_S\bigl(\alpha(t), \beta(w_i)\bigr)\leq 2L(L_1+1)+L_1$. Let $s=w_i$ then $\bar{\gamma}$ is a path from $\alpha(s_1)$ to $\alpha(s_2)$ which lies outside the open ball with radius $\bigl(2r-2L(L_1+1)-2L_1-1\bigr)$ about $\alpha(s)$. Thus, $\bar{\gamma}$ lies outside the open ball with radius $r$ about $\alpha(s)$ by the choice of $r$. We now show that $r\leq \abs{s-s_1}\leq (2L+1)r$. In fact, \begin{align*} \abs{s-s_1}&= d_S\bigl(\alpha(s), \alpha(s_1)\bigr)\\&\leq d_S\bigl(\alpha(s),\beta(t)\bigr)+d_S\bigl(\beta(t),\beta(t-2Lr)\bigr)+d_S\bigl(\beta(t-2Lr),\alpha(s_1)\bigr)\\&\leq 2L(L_1+1)+L_1+2Lr+L_1\\&\leq(2L+1)r
\end{align*}

and \begin{align*} \abs{s-s_1}&= d_S\bigl(\alpha(s), \alpha(s_1)\bigr)\\&\geq d_S\bigl(\beta(t),\beta(t-2Lr)\bigr)-d_S(\bigl(\beta(t),\alpha(s)\bigr)-d_S\bigl(\beta(t-2Lr,\alpha(s_1)\bigr)\\&\geq \frac{1}{L}d_T\beta(t),\beta(t-2Lr)\bigr)-1-2L(L_1+1)-L_1-L_1\\&\geq 2r-2L(L_1+1)-2L_1-1\geq r\end{align*}

Similarly, $r\leq \abs{s-s_2}\leq (2L+1)r$
 
Let $s_3$ be a number between $s$ and $s_1$ such that $\abs{s-s_3}=r$. Let $s_4$ be a number between $s$ and $s_2$ such that $\abs{s-s_4}=r$. Let $\alpha_1$ be a subsegment of $\alpha$ connecting $\alpha(s_1)$ and $\alpha(s_3)$. Let $\alpha_2$ be a subsegment of $\alpha$ connecting $\alpha(s_2)$ to $\alpha(s_4)$. Then $\alpha_1$ and $\alpha_2$ both lie outside the open ball with radius $r$ about $\alpha(s)$ with respect metric $d_S$. Moreover, the length of $\alpha_1$ and $\alpha_2$ are both bounded by $(2L+1)r$.

Let $\alpha_3=\alpha_1\cup\bar{\gamma}\cup \alpha_2$ then $\alpha_3$ is a path from $\alpha(s_3)$ to $\alpha(s_4)$ which lies outside the open ball with radius $r$ about $\alpha(s)$.

 Thus,
\[g(r)\leq \abs{\alpha_3}\leq \abs{\alpha_1}+\abs{\bar{\gamma}}+\abs{\alpha_2}\leq (2L+1)r+\abs{\gamma}+2L_1+(2L+1)r\leq \abs{\gamma}+4(L+1)r.\]
Therefore, $g(r)\leq f(Lr)+4(L+1)r$, which implies that $g\preceq f$.
\end{proof}

\begin{figure}
\begin{tikzpicture}[scale=0.7]

\draw (-14,1) node[circle,fill,inner sep=1pt, color=black](1){} -- (-13,2) node[circle,fill,inner sep=1pt, color=black](1){}-- (-12,1) node[circle,fill,inner sep=1pt, color=black](1){}-- (-13,0) node[circle,fill,inner sep=1pt, color=black](1){} -- (-14,1) node[circle,fill,inner sep=1pt, color=black](1){}; 

\draw (-14,1) node[circle,fill,inner sep=1pt, color=black](1){} -- (-14,3) node[circle,fill,inner sep=1pt, color=black](1){} -- (-12,3) node[circle,fill,inner sep=1pt, color=black](1){} --(-12,1) node[circle,fill,inner sep=1pt, color=black](1){};

\node at (-13,4) {$\Omega_1$};

\draw (-11,1) node[circle,fill,inner sep=1pt, color=black](1){} -- (-10,2) node[circle,fill,inner sep=1pt, color=black](1){}-- (-9,1) node[circle,fill,inner sep=1pt, color=black](1){}-- (-10,0) node[circle,fill,inner sep=1pt, color=black](1){} -- (-11,1) node[circle,fill,inner sep=1pt, color=black](1){}; 

\draw (-10,2) node[circle,fill,inner sep=1pt, color=black](1){} -- (-8,1) node[circle,fill,inner sep=1pt, color=black](1){}-- (-10,0) node[circle,fill,inner sep=1pt, color=black](1){};

\draw (-11,1) node[circle,fill,inner sep=1pt, color=black](1){} -- (-10,-1) node[circle,fill,inner sep=1pt, color=black](1){}-- (-9,1) node[circle,fill,inner sep=1pt, color=black](1){};

\draw (-11,1) node[circle,fill,inner sep=1pt, color=black](1){} -- (-11,3) node[circle,fill,inner sep=1pt, color=black](1){}-- (-8,3) node[circle,fill,inner sep=1pt, color=black](1){} -- (-8,1) node[circle,fill,inner sep=1pt, color=black](1){};

\node at (-9.5,4) {$\Omega_2$};

\draw (-7,1) node[circle,fill,inner sep=1pt, color=black](1){} -- (-6,2) node[circle,fill,inner sep=1pt, color=black](1){}-- (-5,1) node[circle,fill,inner sep=1pt, color=black](1){}-- (-6,0) node[circle,fill,inner sep=1pt, color=black](1){} -- (-7,1) node[circle,fill,inner sep=1pt, color=black](1){}; 

\draw (-6,2) node[circle,fill,inner sep=1pt, color=black](1){} -- (-4,1) node[circle,fill,inner sep=1pt, color=black](1){}-- (-6,0) node[circle,fill,inner sep=1pt, color=black](1){};

\draw (-6,2) node[circle,fill,inner sep=1pt, color=black](1){} -- (-3,1) node[circle,fill,inner sep=1pt, color=black](1){}-- (-6,0) node[circle,fill,inner sep=1pt, color=black](1){};

\draw (-4,1) node[circle,fill,inner sep=1pt, color=black](1){} -- (-5,-1) node[circle,fill,inner sep=1pt, color=black](1){}-- (-6,-1) node[circle,fill,inner sep=1pt, color=black](1){};

\draw (-7,1) node[circle,fill,inner sep=1pt, color=black](1){} -- (-6,-1) node[circle,fill,inner sep=1pt, color=black](1){}-- (-5,1) node[circle,fill,inner sep=1pt, color=black](1){};

\draw (-7,1) node[circle,fill,inner sep=1pt, color=black](1){} -- (-7,3) node[circle,fill,inner sep=1pt, color=black](1){}-- (-3,3) node[circle,fill,inner sep=1pt, color=black](1){} -- (-3,1) node[circle,fill,inner sep=1pt, color=black](1){};

\node at (-5,4) {$\Omega_3$};

\draw (-2,1) node[circle,fill,inner sep=1.5pt, color=black](1){};
\draw (-1.5,1) node[circle,fill,inner sep=1.5pt, color=black](1){};
\draw (-1,1) node[circle,fill,inner sep=1.5pt, color=black](1){};
\draw (-0.5,1) node[circle,fill,inner sep=1.5pt, color=black](1){};
\draw (0,1) node[circle,fill,inner sep=1.5pt, color=black](1){};

\draw (1,1) node[circle,fill,inner sep=1pt, color=black](1){} -- (2,2) node[circle,fill,inner sep=1pt, color=black](1){}-- (3,1) node[circle,fill,inner sep=1pt, color=black](1){}-- (2,0) node[circle,fill,inner sep=1pt, color=black](1){} -- (1,1) node[circle,fill,inner sep=1pt, color=black](1){}; 

\draw (2,2) node[circle,fill,inner sep=1pt, color=black](1){} -- (4,1) node[circle,fill,inner sep=1pt, color=black](1){}-- (2,0) node[circle,fill,inner sep=1pt, color=black](1){};

\draw (2,2) node[circle,fill,inner sep=1pt, color=black](1){} -- (5,1) node[circle,fill,inner sep=1pt, color=black](1){}-- (2,0) node[circle,fill,inner sep=1pt, color=black](1){};

\draw (2,2) node[circle,fill,inner sep=1pt, color=black](1){} -- (9,1) node[circle,fill,inner sep=1pt, color=black](1){}-- (2,0) node[circle,fill,inner sep=1pt, color=black](1){};

\draw (2,2) node[circle,fill,inner sep=1pt, color=black](1){} -- (8,1) node[circle,fill,inner sep=1pt, color=black](1){}-- (2,0) node[circle,fill,inner sep=1pt, color=black](1){};

\draw (4,1) node[circle,fill,inner sep=1pt, color=black](1){} -- (3,-1) node[circle,fill,inner sep=1pt, color=black](1){}-- (2,-1) node[circle,fill,inner sep=1pt, color=black](1){};

\draw (5,1) node[circle,fill,inner sep=1pt, color=black](1){} -- (4,-1) node[circle,fill,inner sep=1pt, color=black](1){}-- (3,-1) node[circle,fill,inner sep=1pt, color=black](1){};

\draw (8,1) node[circle,fill,inner sep=1pt, color=black](1){} -- (7,-1) node[circle,fill,inner sep=1pt, color=black](1){}-- (6,-1) node[circle,fill,inner sep=1pt, color=black](1){};

\draw[densely dotted] (7,1) -- (6.5,0);

\draw (6.5,0) -- (6,-1);

\draw[densely dotted] (4,-1) node[circle,fill,inner sep=1pt, color=black](1){}-- (6,-1) node[circle,fill,inner sep=1pt, color=black](1){};

\draw[densely dotted] (5.5,1) -- (7,1) node[circle,fill,inner sep=1pt, color=black](1){};

\draw (1,1) node[circle,fill,inner sep=1pt, color=black](1){} -- (2,-1) node[circle,fill,inner sep=1pt, color=black](1){}-- (3,1) node[circle,fill,inner sep=1pt, color=black](1){};

\node at (2,-0.25) {$b_0$};

\node at (2,2.25) {$a_0$};

\node at (0.75,1) {$b_1$};

\node at (3.3,1) {$a_1$};

\node at (4.3,1) {$a_2$};

\node at (5.3,1) {$a_3$};

\node at (8.4,0.6) {$a_{d-1}$};

\node at (9.4,1) {$a_d$};

\node at (2,-1.4) {$b_2$};

\node at (3,-1.4) {$b_3$};

\node at (4,-1.4) {$b_4$};

\node at (6,-1.4) {$b_{d-1}$};

\node at (7,-1.4) {$b_d$};

\node at (1,3.4) {$c_1$};

\node at (9,3.4) {$c_2$};

\draw (1,1) node[circle,fill,inner sep=1pt, color=black](1){} -- (1,3) node[circle,fill,inner sep=1pt, color=black](1){}-- (9,3) node[circle,fill,inner sep=1pt, color=black](1){} -- (9,1) node[circle,fill,inner sep=1pt, color=black](1){};

\node at (5,4) {$\Omega_d$};

\end{tikzpicture}

\caption{}
\label{afirst}
\end{figure}

The following theorem shows that the divergence spectrum of a peripheral subgroup is a part of the divergence spectrum of the whole relatively hyperbolic group. The proof of the theorem follows directly from Proposition \ref{p1} and Proposition \ref{p2}.
\begin{thm}
\label{tm1}
Let $(G,\PP)$ be a finitely generated relatively hyperbolic group. Let $P$ be a peripheral subgroup in $\PP$ such that the divergence spectrum of $P$ only contains subexponential functions. Then the divergence spectrum of $P$ is a subset of the divergence spectrum of $G$. 
\end{thm}

\section{Application to right-angled Coxeter groups}
\label{spora}
In this section, we construct an infinite collection of right-angled Coxeter groups which are also relatively hyperbolic groups. We apply the properties of the divergence spectra of relatively hyperbolic groups to show that all groups in this collection are not pairwise quasi-isometric.

We first recall definitions and results concerning right-angled Coxeter groups and their associated Davis complexes.
\begin{defn}
Given a finite, simplicial graph $\Gamma$, the associated right-angled Coxeter group $G_{\Gamma}$ has generating set $S$ the vertices of $\Gamma$, and relations $s^2 = 1$ for all $s$ in $S$ and $st = ts$ whenever $s$ and $t$ are adjacent vertices. If $T\subset S$, then the subgroup $H_T$ of $G_{\Gamma}$ generated by $T$ is a right-angled Coxeter group which associates to the graph $\Gamma_1$, where $\Gamma_1$ is a full subgraph of $\Gamma$ with the set of vertices $T$. We call $H_T$ a \emph{special subgroup}.



\end{defn}

\begin{figure}
\begin{tikzpicture}[scale=0.7]

\draw (-14,1) node[circle,fill,inner sep=1pt, color=black](1){} -- (-13,2) node[circle,fill,inner sep=1pt, color=black](1){}-- (-12,1) node[circle,fill,inner sep=1pt, color=black](1){}-- (-13,0) node[circle,fill,inner sep=1pt, color=black](1){} -- (-14,1) node[circle,fill,inner sep=1pt, color=black](1){}; 


\node at (-13,3) {$\Gamma_1$};

\draw (-11,1) node[circle,fill,inner sep=1pt, color=black](1){} -- (-10,2) node[circle,fill,inner sep=1pt, color=black](1){}-- (-9,1) node[circle,fill,inner sep=1pt, color=black](1){}-- (-10,0) node[circle,fill,inner sep=1pt, color=black](1){} -- (-11,1) node[circle,fill,inner sep=1pt, color=black](1){}; 

\draw (-10,2) node[circle,fill,inner sep=1pt, color=black](1){} -- (-8,1) node[circle,fill,inner sep=1pt, color=black](1){}-- (-10,0) node[circle,fill,inner sep=1pt, color=black](1){};

\draw (-11,1) node[circle,fill,inner sep=1pt, color=black](1){} -- (-10,-1) node[circle,fill,inner sep=1pt, color=black](1){}-- (-9,1) node[circle,fill,inner sep=1pt, color=black](1){};


\node at (-9.5,3) {$\Gamma_2$};

\draw (-7,1) node[circle,fill,inner sep=1pt, color=black](1){} -- (-6,2) node[circle,fill,inner sep=1pt, color=black](1){}-- (-5,1) node[circle,fill,inner sep=1pt, color=black](1){}-- (-6,0) node[circle,fill,inner sep=1pt, color=black](1){} -- (-7,1) node[circle,fill,inner sep=1pt, color=black](1){}; 

\draw (-6,2) node[circle,fill,inner sep=1pt, color=black](1){} -- (-4,1) node[circle,fill,inner sep=1pt, color=black](1){}-- (-6,0) node[circle,fill,inner sep=1pt, color=black](1){};

\draw (-6,2) node[circle,fill,inner sep=1pt, color=black](1){} -- (-3,1) node[circle,fill,inner sep=1pt, color=black](1){}-- (-6,0) node[circle,fill,inner sep=1pt, color=black](1){};

\draw (-4,1) node[circle,fill,inner sep=1pt, color=black](1){} -- (-5,-1) node[circle,fill,inner sep=1pt, color=black](1){}-- (-6,-1) node[circle,fill,inner sep=1pt, color=black](1){};

\draw (-7,1) node[circle,fill,inner sep=1pt, color=black](1){} -- (-6,-1) node[circle,fill,inner sep=1pt, color=black](1){}-- (-5,1) node[circle,fill,inner sep=1pt, color=black](1){};


\node at (-5,3) {$\Gamma_3$};

\draw (-2,1) node[circle,fill,inner sep=1.5pt, color=black](1){};
\draw (-1.5,1) node[circle,fill,inner sep=1.5pt, color=black](1){};
\draw (-1,1) node[circle,fill,inner sep=1.5pt, color=black](1){};
\draw (-0.5,1) node[circle,fill,inner sep=1.5pt, color=black](1){};
\draw (0,1) node[circle,fill,inner sep=1.5pt, color=black](1){};

\draw (1,1) node[circle,fill,inner sep=1pt, color=black](1){} -- (2,2) node[circle,fill,inner sep=1pt, color=black](1){}-- (3,1) node[circle,fill,inner sep=1pt, color=black](1){}-- (2,0) node[circle,fill,inner sep=1pt, color=black](1){} -- (1,1) node[circle,fill,inner sep=1pt, color=black](1){}; 

\draw (2,2) node[circle,fill,inner sep=1pt, color=black](1){} -- (4,1) node[circle,fill,inner sep=1pt, color=black](1){}-- (2,0) node[circle,fill,inner sep=1pt, color=black](1){};

\draw (2,2) node[circle,fill,inner sep=1pt, color=black](1){} -- (5,1) node[circle,fill,inner sep=1pt, color=black](1){}-- (2,0) node[circle,fill,inner sep=1pt, color=black](1){};

\draw (2,2) node[circle,fill,inner sep=1pt, color=black](1){} -- (9,1) node[circle,fill,inner sep=1pt, color=black](1){}-- (2,0) node[circle,fill,inner sep=1pt, color=black](1){};

\draw (2,2) node[circle,fill,inner sep=1pt, color=black](1){} -- (8,1) node[circle,fill,inner sep=1pt, color=black](1){}-- (2,0) node[circle,fill,inner sep=1pt, color=black](1){};

\draw (4,1) node[circle,fill,inner sep=1pt, color=black](1){} -- (3,-1) node[circle,fill,inner sep=1pt, color=black](1){}-- (2,-1) node[circle,fill,inner sep=1pt, color=black](1){};

\draw (5,1) node[circle,fill,inner sep=1pt, color=black](1){} -- (4,-1) node[circle,fill,inner sep=1pt, color=black](1){}-- (3,-1) node[circle,fill,inner sep=1pt, color=black](1){};

\draw (8,1) node[circle,fill,inner sep=1pt, color=black](1){} -- (7,-1) node[circle,fill,inner sep=1pt, color=black](1){}-- (6,-1) node[circle,fill,inner sep=1pt, color=black](1){};

\draw[densely dotted] (7,1) -- (6.5,0);

\draw (6.5,0) -- (6,-1);

\draw[densely dotted] (4,-1) node[circle,fill,inner sep=1pt, color=black](1){}-- (6,-1) node[circle,fill,inner sep=1pt, color=black](1){};

\draw[densely dotted] (5.5,1) -- (7,1) node[circle,fill,inner sep=1pt, color=black](1){};

\draw (1,1) node[circle,fill,inner sep=1pt, color=black](1){} -- (2,-1) node[circle,fill,inner sep=1pt, color=black](1){}-- (3,1) node[circle,fill,inner sep=1pt, color=black](1){};

\node at (2,-0.25) {$b_0$};

\node at (2,2.25) {$a_0$};

\node at (0.75,1) {$b_1$};

\node at (3.3,1) {$a_1$};

\node at (4.3,1) {$a_2$};

\node at (5.3,1) {$a_3$};

\node at (8.4,0.6) {$a_{d-1}$};

\node at (9.4,1) {$a_d$};

\node at (2,-1.4) {$b_2$};

\node at (3,-1.4) {$b_3$};

\node at (4,-1.4) {$b_4$};

\node at (6,-1.4) {$b_{d-1}$};

\node at (7,-1.4) {$b_d$};




\node at (5,3) {$\Gamma_d$};

\end{tikzpicture}

\caption{}
\label{asecond}
\end{figure}

\begin{defn}
Given a nontrivial, connected, finite, simplicial, triangle-free graph $\Gamma$ with the set $S$ of vertices, we may define \emph{the Davis complex} $\Sigma=\Sigma_{\Gamma}$ to be the Cayley 2--complex for the presentation of the Coxeter group $G_{\Gamma}$, in which all disks bounded by a loop with label $s^2$ for $s$ in $S$ have been shrunk to an unoriented edge with label $s$. Then the vertex set of $\Sigma$ is $G_{\Gamma}$ and the 1-skeleton of $\Sigma$ is the Cayley graph $C_{\Gamma}$ of $G_{\Gamma}$ with respect to the generating set $S$. Since all relators in this presentation other than $s^2 = 1$ are of the form $stst = 1$, $\Sigma$ is a square complex. 
\end{defn}

\begin{rem}
\label{ra}
The Davis complex $\Sigma_{\Gamma}$ is a $\CAT(0)$ space and the group $G_{\Gamma}$ acts properly and cocompactly on the Davis complex $\Sigma_{\Gamma}$ (see \cite{MR2360474}). Moreover, $\Sigma_{\Gamma}$ and $C_{\Gamma}$ have the geodesic extension property. 

Let $H_T$ be a special subgroup of $G_{\Gamma}$. Then the Cayley graph of $H_T$ (with respect to the generating set $T$) embeds isometrically in $C_{\Gamma}\subset \Sigma_{\Gamma}$. 
\end{rem}




For each $d\geq 2$, let $\Omega_d$ be the graph in Figure \ref{afirst} and $\Gamma_d$ be the full subgraph graph of $\Omega_d$ in Figure \ref{asecond}. We remark that the graphs $\Gamma_d$ were introduced by Dani-Thomas \cite{MR3314816} to study divergence of right-angled Coxeter groups. In \cite{MR3314816}, Dani-Thomas proved that the divergence of $G_{\Gamma_d}$ (or $C_{\Gamma_d}$) is a polynomial $r^d$. We observe that $G_{\Gamma_d}$ is a special subgroup of $G_{\Omega_d}$ and $(G_{\Omega_d}, G_{\Gamma_d})$ is relatively hyperbolic (see \cite[Theorem A']{MR3450952}). The polynomial $r^d$ is a function in the divergence spectrum of $G_{\Gamma_d}$ due to the following proposition. 


\begin{prop}[Proposition 3.19 in \cite{MR3451473}] 
\label{propo}
For each $d\geq 2$, let $C_{\Gamma_d}$ be the 1-skeleton of $\Sigma_{\Gamma_d}$. Let $\alpha_d$ be a bi-infinite geodesic containing $e$ and labeled by $a_db_da_db_d\cdots$. Then the lower divergence of $\alpha_d$ is equivalent to the polynomial of degree $d$. Thus, the polynomial $r^d$ is a function in the divergence spectrum of $G_{\Gamma_d}$.
\end{prop}

We remark that the above proposition was proved in \cite{MR3451473} but the proof is based mostly on the work of Dani-Thomas in \cite{MR3314816}. Before stating the main theorem of this section, we remind the reader that each group $G_{\Omega_d}$ is finitely presented, one-ended, and relatively hyperbolic. Therefore, the group divergence of each group $G_{\Omega_d}$ is exactly exponential (see \cite{Sisto}). However, we will show that two groups $G_{\Omega_d}$ and $G_{\Omega_{d'}}$ have different divergence spectra for each $d\neq d'$ (see the following theorem).

\begin{thm}
Let $d$ and $d'$ be two different positive integers. Then $G_{\Omega_d}$ and $G_{\Omega_{d'}}$ have different divergence spectra. Therefore, they are not quasi-isometric.
\end{thm}
\begin{proof}
We assume that $d'<d$. It is obvious that $r^d$ and $r^{d'}$ are both subexponential. Also, $r^{d'}$ is strictly dominated by $r^d$. By Theorem \ref{tm1} and Proposition \ref{propo}, $r^d$ is a function in the divergence spectrum of $G_{\Omega_d}$. By the Theorem \ref{tm2}, Remark \ref{ra}, the divergence spectrum of $G_{\Omega_{d'}}$ only contains functions which are dominated by $r^{d'}$ or at least exponential. Thus, the divergence spectrum of $G_{\Omega_{d'}}$ does not contain $r^d$. Then $G_{\Omega_d}$ and $G_{\Omega_{d'}}$ have different divergence spectra. Therefore, they are not quasi-isometric. 
\end{proof}

\begin{rem}
\label{aw}
We remark that there is an alternate (shorter) way to classify groups $G_{\Omega_d}$ without computing their divergence spectra. Dani-Thomas \cite{MR3314816} proved that the divergence of each peripheral subgroup $G_{\Gamma_d}$ is polynomials of degree $d$. Therefore, each group $G_{\Gamma_d}$ is not relatively hyperbolic with respect to any collection of proper subgroups by Theorem 1.3 \cite{Sisto}. Moreover, two groups $G_{\Gamma_{d_1}}$ and $G_{\Gamma_{d_2}}$ ($d_1\neq d_2$) are not quasi-isometric because they have different divergence functions. Therefore, two groups $G_{\Omega_{d_1}}$ and $G_{\Omega_{d_2}}$ ($d_1\neq d_2$) are also not quasi-isometric by Theorem 4.1 \cite{MR2501302}. 
\end{rem}

\section{Morse boundaries}
\label{mbdr}
In this section, we review the concept of Morse boundary in \cite{MC}. We will show that for each finitely generated relatively hyperbolic group $(G,\PP)$, the inclusion map of each peripheral subgroup $P$ induces a topological embedding of the Morse boundary of $P$ into the Morse boundary of $G$. 

\begin{defn}
Let $\mathcal{M}$ be the set of all Morse gauges. We put a partial ordering on $\mathcal{M}$ so that for two Morse gauges $N, N' \in \mathcal{M}$, we say $N \leq N'$ if and only if $N(K,L) \leq N'(K, L)$ for all $K, L$.
\end{defn}

\begin{defn}
Let $X$ be a proper geodesic space. The \emph{Morse boundary} of $X$ with basepoint p, denoted $\partial_M {X_p}$, is defined to be the set of all equivalence classes of Morse geodesic rays in $X$ with initial point $p$, where two rays $\alpha,\alpha'\!:[0,\infty)\rightarrow X$ are equivalent if there exists a constant $K$ such that $d_X\bigl(\alpha(t), \alpha′(t)\bigr) < K$ for all $t > 0$. We denote the equivalence class of a ray $\alpha$ in $\partial_M {X_p}$ by $[\alpha]$.

On $\partial_M {X_p}$, we build a topology as follows :

Consider the subset of the Morse boundary 
\[\partial_M^N {X_p}=\set{x}{\text{The class $x$ contains an $N$--Morse geodesic ray $\alpha$ with $\alpha(0)=p$}}.\]

We define convergence in $\partial_M^N {X_p}$ by: $x_n \rightarrow x$ as $n \rightarrow \infty$ if and only if there exists $N$--Morse geodesic rays $\alpha_n$ with $\alpha_n(0) = p$ and $[\alpha_n] = x_n$ such that every subsequence of ${\alpha_n}$ contains a subsequence that converges uniformly on compact sets to a geodesic ray $\alpha$ with $[\alpha]=x$. The closed subsets $F$ in $\partial_M {X_p}$ are those satisfying the condition
\[\bigl[\{x_n\} \subset F \text{ and } x_n \rightarrow x\bigr] \implies x\in F.\]
We equip the Morse boundary $\partial_M {X_p}$ with the direct limit topology $$\partial_M {X_p}=\lim_{\overrightarrow{\mathcal{M}}} \partial_M^N {X_p}.$$
\end{defn}

\begin{rem}
The direct limit topology on $\partial_M {X_p}$ is independent of basepoint $p$ (see Proposition 3.5 in \cite{MC}). Therefore, we can assume the
basepoint is fixed, suppress it from the notation and write $\partial_M {X}$. Moreover, the Morse boundary is a quasi-isometry invariant (see Proposition 3.7 \cite{MC}). Therefore, we define the \emph{Morse boundary} of a finitely generated group $G$, denoted $\partial_M G$, as the Morse boundary of its Cayley graph.

We define \emph{an action} of $G$ on $\partial_M G$ as follows. For each element $g$ in $G$ and $[\alpha]$ in $\partial_M G$, $g[\alpha]=[\beta]$, where $\alpha$ and $\beta$ are two rays at the basepoint in some Cayley graph of $G$ such that the Hausdorff distance between $g\alpha$ and $\beta$ is finite.
\end{rem}

\begin{defn}
Let $X$ and $Y$ be proper geodesic metric spaces and $p \in X$, $p' \in Y$.
We say that $f\!:\partial_M {X_p} \rightarrow \partial_M {Y_{p'}}$ is \emph{Morse preserving} if given $N$ in $\mathcal{M}$ there exists
an $N'$ in $\mathcal{M}$ such that $f$ injectively maps $\partial_M^N {X_p}$ to $\partial_M^{N'} {Y_{p'}}$.

Let $G$ and $H$ be two finitely generated groups. Let $\phi\!:H\rightarrow G$ be a quasi-isometric embedding. Let $\Phi: \Gamma(H,T)\rightarrow \Gamma(G,S)$ be an extension of $\phi$ for some (any) generating sets $T$ and $S$. Then $\phi$ induces a \emph{Morse preserving map} if the induced map $\partial_M\Phi$ is Morse preserving. We call $\partial_M\Phi$ the Morse preserving map induced by $\phi$, denoted $\partial_M\phi$
\end{defn}

\begin{prop}[Proposition 4.2 in \cite{MC}]
\label{Corte} 
If $\Phi\!: X \rightarrow Y$ is a quasi-isometric embedding that induces a Morse preserving map, then the induced map $\partial_M \Phi: \partial_M X \rightarrow \partial_M Y$ is a topological embedding.
\end{prop}

The following theorem is a corollary of Proposition \ref{p1} and Proposition ~\ref{Corte}.

\begin{thm}
Let $(G,\PP)$ be a finitely generated relatively hyperbolic group. Then for each peripheral subgroup $P$ in $\PP$ the inclusion $i_P\!:P\hookrightarrow G$ induces a Morse preserving map. Therefore, $\partial_M i_P: \partial_M P \rightarrow \partial_M G$ is a topological embedding.
\end{thm}

\begin{rem}
By the above theorem, we can consider the Morse boundary of each peripheral subgroup $P$ in a finitely generated relatively hyperbolic group $(G,\PP)$ is a subspace of the Morse boundary of $G$. 

For each peripheral left coset $gP$, we define its boundary in $\partial_M G$, denoted $\partial_M {gP}$, to be the set of all $[\gamma]\in\partial_M G$ where $\gamma \subset N_R(gP)$ for some $R$. Each element in $\partial_M gP$ for some peripheral left coset $gP$ is said to be a \emph{peripheral limit point}. Each element $x$ in $\partial_M G$ that does not lie in any $\partial_M gP$ is said to be a \emph{non-peripheral limit point}. It is not hard to see that $\partial_M gP=g\bigl(\partial_M P\bigr)$. Two peripheral limit points are said to be of \emph{the same type} if they both lie in $\partial_M gP$ for some peripheral left coset $gP$.
\end{rem}

\section{Bowditch boundaries and some connection to Morse boundaries}
\label{bbvsmb}
In this section, we review the concept of Bowditch boundary in \cite{MR2922380}. We will show a connection between Morse boundary and Bowditch boundary for a finitely generated relatively hyperbolic group $(G,\PP)$.

Suppose $(G,\PP)$ is relatively hyperbolic with a finite generating set $S$. Let $\hat\Gamma (G,S,\PP)$ be a coned-off Cayley graph of $(G,\PP)$. If everything is clear from context, we can use the notation $\hat{\Gamma}$ instead of using $\hat\Gamma (G,S,\PP)$. Since $(G,\PP)$ is relatively hyperbolic, $\hat{\Gamma}$ is a connected, fine, and hyperbolic graph. Let $V_{\infty}(\hat{\Gamma})$ be the set of all peripheral vertices of $\hat{\Gamma}$ and $\partial \hat{\Gamma}$ be the usual hyperbolic boundary of $\hat{\Gamma}$. The set $\Delta_{\infty}(\hat{\Gamma})=V_{\infty}(\hat{\Gamma})\bigcup\partial \hat{\Gamma}$ is \emph{the infinite closure} of $\hat{\Gamma}$. The \emph{Bowditch boundary} of $(G,\PP)$, denoted $\partial (G,\PP)$, is defined as $\Delta_{\infty}(\hat{\Gamma})$ and we put a topology on $\Delta_{\infty}(\hat{\Gamma})$ as follows:

For each $a\in\Delta_{\infty}(\hat{\Gamma})$ and $A$ a finite set of $V(\hat{\Gamma})$ that does not contain $a$, we define $M(a,A)$ to be the set of points $b\in\Delta_{\infty}(\hat{\Gamma})$ such that there is at least one geodesic $\alpha$ in $\hat{\Gamma}$ from $a$ to $b$ that does not meet $A$. Then the collection of all such sets $M(a,A)$ form a basis of a topology on $\Delta_{\infty}(\hat{\Gamma})$ (see \cite{MR2922380}). We define a set $U\subset \Delta_{\infty}(\hat{\Gamma})$ to be open if for all $a\in U$, there is a finite subset $A\subset V(\hat{\Gamma})$ that does not contain $a$ such that $M(a,A)\subset U$.

\begin{rem}
Bowditch has shown that the Bowditch boundary does not depend on the choice of finite generating set (see \cite{MR2922380}).

Each element in $V_{\infty}(\hat{\Gamma})$ is said to be a \emph{parabolic point} in $\partial (G,\PP)$ and each element in $\partial \hat{\Gamma}$ is said to be a \emph{non-parabolic point} in $\partial (G,\PP)$.
\end{rem}

We now state some topological properties of $\partial (G,\PP)$ from \cite{MR2922380}. For each $\lambda\geq 1, c\geq 0$, $a\in\Delta_{\infty}(\hat{\Gamma})$ and a finite set $A$ of $V(\hat{\Gamma})$ that does not contain $a$, we define $M_{(\lambda,c)}(a,A)$ to be the set of points $b\in\Delta_{\infty}(\hat{\Gamma})$ such that there is at least one $(\lambda,c)$--quasi-geodesic arc $\alpha$ in $\hat{\Gamma}$ from $a$ to $b$ that does not meet $A$.

\begin{lem}[\cite{MR2922380}] Let $(G,\PP)$ be a finitely generated relatively hyperbolic group. Then:
\begin{enumerate}
\item For each $\lambda\geq 1,c\geq 0$, the collection of all sets of the form $M_{(\lambda,c)}(a,A)$ forms a basis for the topology of the Bowditch boundary $\partial (G,\PP)$.
\item The Bowditch boundary $\partial (G,\PP)$ is compact and Hausdorff. 
\end{enumerate}
\end{lem}

The following lemma is a well-known property of geodesics in $\delta$--hyperbolic spaces.

\begin{lem}
\label{l14}
For each choice of positive constants $\delta$ and $\sigma$, there is a positive number $R=R(\delta,\sigma)$ such that the following holds. Let $\alpha$ and $\alpha'$ be two equivalent geodesic rays in a $\delta$--hyperbolic space such that $d(\alpha_+,\alpha'_+)\leq\sigma$ or let $\alpha$ and $\alpha'$ be two geodesic segments such that $d(\alpha_+,\alpha'_+)\leq\sigma$ and $d(\alpha_-,\alpha'_-)\leq\sigma$. Then the Hausdorff distance between them is at most $R$.
\end{lem}

We now review some geometric connections between the Cayley graph and the coned off Cayley graph. From now, we denote the metric in $\Gamma(G,S)$ by $d_S$, the metric in $\hat{\Gamma}(G,S,\PP)$ by $d$. 
\begin{lem}[Lemma 4.13 in \cite{MR3143594}]
\label{l8}
Let $c$ and $c'$ be two equivalent geodesic rays in $\hat{\Gamma}(G,S,\PP)$ (i.e. the Hausdorff distance between $c$ and $c'$ is finite with respect to d) with the same initial point $h_0$. Suppose that $(g_n)$ and $(g'_n)$ are the sequences of all $G$--vertices of $c$ and $c'$ respectively. Then the Hausdorff distance between $(g_n)$ and $(g'_n)$ is finite with respect to the metric $d_S$.
\end{lem}

\begin{lem}[Lemma 4.16 in \cite{MR3143594}] 
\label{l5}
There is a positive constant $A$ such that the following holds. Let $\alpha$ be a geodesic ray in $\Gamma(G,S)$ such that $\alpha$ is not contained in $N_R(gP)$ for any peripheral left coset $gP$ and any positive number $R$. Then there is a geodesic ray $c$ in $\hat{\Gamma}(G,S,\PP)$ such that $c$ and $\alpha$ have the same initial point, all $G$--vertices of $c$ lie in the $A$--neighborhood of $\alpha$ with respect to the metric $d_S$ and $\alpha$ lies in the $A$--neighborhood of the $c$ with respect to the metric $d$.
\end{lem}

\begin{lem}[Lemma 4.17 in \cite{MR3143594}] 
\label{l15}
There is a positive constant $B$ such that the following holds. Let $\alpha$ be a geodesic ray in $\Gamma(G,S)$ such that $\alpha$ is contained in $N_R(g^*P^*)$ for some peripheral left coset $g^*P^*$ and some positive number $R$. Let $c$ be a geodesic in $\hat{\Gamma}(G,S,\PP)$ that connects the initial point of $\alpha$ and $v_{g^*P^*}$. Then all $G$--vertices of $c$ lie in the $B$--neighborhood of $\alpha$ with respect to the metric $d_S$ and $\alpha$ lies in the $B$--neighborhood of the $c$ with respect to the metric $d$.
\end{lem}

Now, we build the map $f\colon \partial_M \Gamma\to \Delta_{\infty}(\hat{\Gamma})$ between the Morse boundary of the Cayley graph $\Gamma(G,S)$ and the infinite hyperbolic closure of $\hat{\Gamma}(G,S,\PP)$ as follows:

We fix $e$ in $\Gamma(G,S)$ as a basepoint when referring to equivalence classes of rays in $\partial_M \Gamma$. Let $[\alpha]$ be a point in $\partial_M \Gamma$. If $[\alpha]\in \bigcup_{gP\in\Pi} \partial_M{gP}$, then there is a unique peripheral left coset $g_0P_0$ such that $[\alpha]\in \partial_M{g_0P_0}$. We define $f\bigl([\alpha]\bigr)=v_{g_0P_0}$. If $[\alpha]\notin \bigcup_{gP\in\Pi} \partial_M {gP}$, then there is a geodesic ray $c$ in $\hat{\Gamma}(G,S,\PP)$ such that $c$ and $\alpha$ have the same initial point, all $G$--vertices of $c$ lie in the some neighborhood of $\alpha$ with respect to the metric $d_S$ and $\alpha$ lies in some neighborhood of the $c$ with respect to the metric $d$ (see Lemma \ref{l5}). We define $f\bigl([\alpha]\bigr)=[c]$.

\begin{lem}
The map $f$ is well-defined and $G$--equivariant.
\end{lem}

\begin{proof}
Suppose $[\alpha_1]$= $[\alpha_2]$ in $\partial_M \Gamma$, where $\alpha_1$ and $\alpha_2$ are two Morse geodesic rays in $\Gamma(G,S)$ with the same endpoint $e$. If one of them belongs to $\bigcup_{gP\in\Pi} \partial_M{gP}$\ then there is a unique peripheral left coset $g_0P_0$ such that $[\alpha_1] = [\alpha_2] \in \partial_M{g_0P_0}$. Therefore, $f\bigl([\alpha_1]\bigr)= f\bigl([\alpha_2]\bigr) = v_{g_0P_0}$. Suppose that [$\alpha_1$]= [$\alpha_2$] lies in $\partial_M \Gamma-\bigcup_{gP\in\Pi}\partial_M{gP}$. Let $c_1$ be a geodesic ray in $\hat{\Gamma}(G,S,\PP)$ such that $c_1$ and $\alpha_1$ have the same initial point, all $G$--vertices of $c_1$ lie in the some neighborhood of $\alpha_1$ with respect to the metric $d_S$ and $\alpha_1$ lies in some neighborhood of the $c_1$ with respect to the metric $d$. We choose a similar geodesic $c_2$ in $\hat{\Gamma}(G,S,\PP)$ for $\alpha_2$. This implies that the Hausdorff distance between $c_1$ and $c_2$ is finite with respect to the metric $d$. Therefore, $[c_1]=[c_2]$. This implies that the $f$ is well-defined. Also, the group $G$ acts geometrically on both Cayley graph $\Gamma(G,S)$ and coned-off Cayley graph $\hat{\Gamma}(G,S,\PP)$, $g\bigl(\partial_M hP\bigr)=\partial_M ghP$ and $gv_{hP}=v_{ghP}$ for each element $g$ in $G$ and each peripheral left coset $hP$. Therefore, the map $f$ is $G$--equivariant.
\end{proof}

\begin{lem}
The map $f$ maps $\partial_M\Gamma-\bigcup_{gP\in\Pi}\partial_M{gP}$ injectively into $\partial \hat{\Gamma}$.
\end{lem}

\begin{proof}
Suppose that $f\bigl([\alpha_1]\bigr)=f\bigl([\alpha_2]\bigr)$, where $\alpha_1$ and $\alpha_2$ are two Morse geodesic rays in $\Gamma(G,S)$ with the same endpoint $e$ such that both $\alpha_1$ and $\alpha_2$ do not lie in any finite neighborhood of any peripheral subgroup. Let $c_1$ be a geodesic ray in $\hat{\Gamma}(G,S,\PP)$ such that $c_1$ and $\alpha_1$ have the same initial point, all $G$--vertices of $c_1$ lie in the some neighborhood of $\alpha_1$ with respect to the metric $d_S$ and $\alpha_1$ lies in some neighborhood of the $c_1$ with respect to the metric $d$. We choose a similar geodesic $c_2$ in $\hat{\Gamma}(G,S,\PP)$ for $\alpha_2$. Then, $f\bigl([\alpha_1]\bigr)=[c_1]$ and $f\bigl([\alpha_2]\bigr)=[c_2]$. Therefore, $c_1$ and $c_2$ are two equivalent geodesic rays in $\hat{\Gamma}(G,S,\PP)$. By Lemma \ref{l8}, the Hausdorff distance with respect to the metric $d_S$ between all $G$--vertices of $c_1$ and all $G$--vertices of $c_2$ is finite. Also, all $G$--vertices of $c_1$ lie in the some neighborhood of $\alpha_1$, all $G$--vertices of $c_2$ lie in the some neighborhood of $\alpha_2$ with respect to the metric $d_S$, and $\alpha_1$, $\alpha_2$ are Morse geodesics rays. Therefore, the Hausdorff distance with respect to the metric $d_S$ between $\alpha_1$ and $\alpha_2$ is finite. This implies that $[\alpha_1] = [\alpha_2]$. 
\end{proof}

We now prove the map $f$ is continuous. Since we equip $\partial_M \Gamma$ with the direct limit topology induced by the collection of spaces $\{\partial_M^N \Gamma\}_{N\in\mathcal{M}}$, it is sufficient to show the restriction of $f$ on each $\partial_M^N \Gamma$ is continuous. 

\begin{prop}
The restriction of $f$ on each $\partial_M^N \Gamma$ is continuous.
\end{prop}

\begin{proof}
We prove the restriction of $f$ on each $\partial_M^N \Gamma$ is continuous by showing that the preimages of the closed sets in $\Delta_{\infty}(\hat{\Gamma})$ via the map $f_{|\partial_M^N \Gamma}$ are closed in $\partial_M^N \Gamma$. Let $F$ be an arbitrary closed set in $\Delta_{\infty}(\hat{\Gamma})$. We will prove $f^{-1}_{|\partial_M^N \Gamma}(F)$ is closed in $\partial_M^N \Gamma$ by showing that for each sequence $\{x_n\}$ in $f^{-1}_{|\partial_M^N \Gamma}(F)$ that converges to $x$ in $\partial_M^N \Gamma$, then $x$ lies in $f^{-1}_{|\partial_M^N \Gamma}(F)$ (i.e. $f_{|\partial_M^N \Gamma}(x)$ lies in $F$). Since the sequence $\{x_n\}$ converges to $x$, there exists $N$--Morse geodesic rays $\alpha_n$ with $\alpha_n(0) = e$ and $[\alpha_n] = x_n$ such that every subsequence of ${\alpha_n}$ contains a subsequence that converges uniformly on compact sets to a geodesic ray $\alpha$ with $[\alpha]=x$. By passing to some subsequence we may assume that ${\alpha_n}$ converges uniformly on compact sets to a geodesic ray $\alpha$. We now prove $x$ lies in $f^{-1}_{|\partial_M^N \Gamma}(F)$ by using the closeness of $F$ and showing that the sequence $f(x_n)$ converges to $f(x)$. The following two lemmas will help us finish the proof of this proposition. 
\end{proof}

\begin{lem}
If $x$ is a non-peripheral limit point, then $f(x_n)$ converges to $f(x)$.
\end{lem}

\begin{proof}
Let $M\bigl(f(x),D\bigr)$ be a neighborhood of $f(x)$, where $D$ is a finite subset of $V(\hat{\Gamma})$. We need to prove that there is a positive integer $n_0$ such that $f(x_n)$ lies in $M\bigl(f(x),D\bigr)$ for each $n>n_0$. 

Let $A$ be the constant in Lemma \ref{l5} and $B$ the constant in Lemma \ref{l15}. Let $\sigma =\max\set{d(e,a)}{a\in D}$. Let $R=R(\delta,\sigma)$ be the constant in Lemma \ref{l14}, where $\delta$ is the hyperbolic constant of $\hat{\Gamma}(G,S,\PP)$. Let $c$ be a geodesic in $\hat{\Gamma}(G,S,\PP)$ with the initial point $e$ such that all $G$--vertices of $c$ lie in the $A$--neighborhood of $\alpha$ with respect to the metric $d_S$ and $\alpha$ lies in the $A$--neighborhood of the $c$ with respect to the metric $d$. Then $f(x)=[c]$ by the construction of $f$. Let $g_0$ be a $G$--vertex in $c$ such that $d(g_0,e)\geq 3A+3B+4R+\sigma+2$ and let $t_0$ in $[0,\infty)$ such that $d_S\bigl(\alpha(t_0),g_0\bigr)<A$. Let $n_0$ be a positive number such that $d_S\bigl(\alpha_n(t_0),\alpha(t_0)\bigr)<1$ for each $n>n_0$. We now prove that $f(x_n)$ lies in $M\bigl(f(x),D\bigr)$ for $n>n_0$. In fact, let $\pi$ be a geodesic in $\hat{\Gamma}(G,S,\PP)$ connecting $f(x_n)$ and $f(x)$, we need to prove that $\pi \cap D=\emptyset$.

We first assume that $x_n$ is a non-peripheral limit point. Let $c_n$ be a geodesic in $\hat{\Gamma}(G,S,\PP)$ with the initial point $e$ such that all $G$--vertices of $c_n$ lie in the $A$--neighborhood of $\alpha_n$ with respect to the metric $d_S$ and $\alpha_n$ lies in the $A$--neighborhood of the $c_n$ with respect to the metric $d$. Then $f(x_n)=[c_n]$ by the construction of $f$. Let $g_1$ be a point in $c_n$ such that the distance between $g_1$ and $\alpha_n(t_0)$ is less than $A$ with respect to the metric $d$. Therefore, the distance between $g_0$ and $g_1$ is less than $2A+1$ with respect to the metric $d$. We now assume for the contradiction that $\pi \cap D\neq\emptyset$. Then we could choose $z$ in $\pi\cap D$ such that $d(e,z)=d(e,\pi)\leq \sigma$. We could consider $\pi$ as the union of two rays $\pi^+$, $\pi^-$ with the same initial point $z$ such that $[\pi^+]=f(x)$ and $[\pi^-]=f(x_n)$. 

Choose $z_1$ in $\pi^+$ and $z'_1$ in $\pi^-$ such that $d(g_0, z_1)\leq R$ and $d(g_1, z'_1)\leq R$. Obviously, $z$ lies between $z_1$ and $z'_1$ (i.e., $d(z'_1,z_1)=d(z'_1,z)+d(z,z_1)$). This implies that \[d(z,e)\geq d(g_0,e)-d(g_0,z_1)-d(z_1, z)\geq d(g_0,e)-d(z_1,z)-R\]

and \[d(z,e)\geq d(g_1,e)-d(g_1,z'_1)-d(z'_1, z)\geq d(g_1,e)-d(z'_1,z)-R.\]
Also, \[d(z'_1,z_1)=d(z'_1,z)+d(z,z_1).\] Then, \[2d(z,e)\geq d(g_0,e)+d(g_1,e)-d(z_1,z'_1)-2R.\] 

We have \[d(g_0,e)+d(g_1,e)\geq 2d(g_0,e)-d(g_0,g_1)\geq 2d(g_0,e)-2A-1\]
and \[d(z_1,z'_1)\leq d(z_1,g_0)+d(g_0,g_1)+d(g_1,z'_1)\leq 2A+2R+1.\]
Therefore, \[2d(z,e)\geq 2d(g_0,e)-4A-4R-2>2\sigma.\]
This is a contradiction. Thus, $\pi \cap D=\emptyset$ and $f(x_n)$ lies in $M\bigl(f(x),D\bigr)$.

We now assume that $x_n$ is a peripheral limit point. Then $f(x_n)=v_{gP}$, where $\alpha_n$ lies in some neighborhood of the peripheral left coset $gP$. Let $c_n$ be a geodesic in $\hat{\Gamma}(G,S,\PP)$ connecting $e$ and $v_{gP}$. By Lemma \ref{l15}, all $G$--vertices of $c_n$ lie in the $B$--neighborhood of $\alpha_n$ with respect to the metric $d_S$ and $\alpha_n$ lies in the $B$--neighborhood of the $c_n$ with respect to the metric $d$. By using a similar argument as above, we can prove that $\pi \cap D=\emptyset$ and $f(x_n)$ lies in $M\bigl(f(x),D\bigr)$.
\end{proof}

\begin{lem}
If $x$ is a peripheral limit point, then $f(x_n)$ converges to $f(x)$.
\end{lem}

\begin{proof}
Since $x$ is a peripheral limit point, there is a peripheral coset $g_0P_0$ such that $\alpha$ lies in the $R$--neighborhood of $g_0P_0$ for some $R$. Therefore, $f(x)=v_{g_0P_0}$ by the construction of $f$. Let $M_{(1, R+2)}\bigl(f(x),D\bigr)$ be a neighborhood of $f(x)$, where $D$ is a finite subset of $V(\hat{\Gamma})$ that does not contain $v_{g_0P_0}$. We need to prove that there is a positive integer $n_0$ such that $f(x_n)$ lies in $M_{(1, R+2)}\bigl(f(x),D\bigr)$ for each $n>n_0$.

Let $A$ be the constant in Lemma \ref{l5} and $B$ the constant in Lemma \ref{l15}. Let $D_1$ be the set of all $G$--vertices in $D$ and $D_2$ the set of all peripheral vertex in $D$. Let $r_1 =\max\set{d_S(e,a)}{a\in D_1}$ and $r_2 =\max\set{d_S(e,gP)}{v_{gP}\in D_2}$. Let $C$ be a positive constant such that each geodesic $\gamma$ in $\Gamma(G,S)$ with endpoints in $(A+B+r_2)$--neighborhood of some peripheral left coset $gP$ must lie entirely in $gP$. Let $t_0\geq R+r_1+r_2+A+B+C$ such that $\alpha(t_0)$ lies outside all $(C+1)$--neighborhoods of peripheral left cosets $gP$, where $v_{gP}$ in $D_2$. Let $n_0$ be a positive number such that $d_S\bigl(\alpha_n(t_0),\alpha(t_0)\bigr)<1$ for each $n>n_0$. We now prove that $f(x_n)$ lies in $M_{(1, R+2)}\bigl(f(x),D\bigr)$ for $n>n_0$.

We first assume that $x_n$ is a non-peripheral limit point. Let $c'_n$ be a geodesic in $\hat{\Gamma}(G,S,\PP)$ with the initial point $\alpha_n(t_0)$ such that all $G$--vertices of $c'_n$ lie in the $A$--neighborhood of $\alpha_n\bigl([t_0,\infty)\bigr)$ with respect to the metric $d_S$ and $\alpha_n\bigl([t_0,\infty)\bigr)$ lies in the $A$--neighborhood of the $c'_n$ with respect to the metric $d$. Then $c'_n$ connects $\alpha_n(t_0)$ and $f(x_n)$. We now prove that $c'_n\cap D = \emptyset$.

Assume for the contradiction that $c'_n\cap D \neq \emptyset$. Then there is a $G$--vertex $u$ in $c'_n$ such that $u\in D_1$ or $u\in gP$ for some $v_{gP}\in D_2$. Let $t\in [t_0,\infty)$ such that the distance between $u$ and $\alpha_n(t)$ is less than $A$ with respect to the metric $d_S$. If $u$ lies in $D_1$, then
\[t=d_S\bigl(\alpha_n(t),e\bigr)\leq d_S\bigl(\alpha_n(t),u\bigr)+d_S(u,e)\leq A+r_1<t_0.\]
This is a contradiction. Therefore, we now assume $u$ lies in some peripheral left coset $gP$, where $v_{gP}\in D_2$. Since $\alpha_n(0)$ and $\alpha_n(t)$ both lie in the $(A+B+r_2)$--neighborhood of $gP$, then $\alpha_n\bigl([0,t]\bigr)$ lies entirely in the $C$--neighborhood of $gP$. In particular, $\alpha_n(t_0)$ lies in the $C$--neighborhood of $gP$. Also, the distance between $\alpha_n(t_0)$ and $\alpha(t_0)$ is less than 1 with respect to the metric $d_S$. Therefore, $\alpha(t_0)$ lies in the $(C+1)$--neighborhood of $gP$. This contradicts the choice of $t_0$. Therefore, $c'_n\cap D = \emptyset$. 

Since $d_S\bigl(\alpha_n(t_0),\alpha(t_0)\bigr)<1$ and $d_S\bigl(\alpha(t_0),g_0P_0)<R$, then there is a path $\ell$ in $\Gamma(G,S)$ with length less than $R+1$ connecting $\alpha_n(t_0)$ and some point $g^*$ in $g_0P_0$. Since $\ell$ is a path in $\Gamma(G,S)$, then $\ell\cap D_2\neq \emptyset$ obviously. If $\ell\cap D_2 \neq \emptyset$, then $t_0=d_S\bigl(\alpha_n(t_0),e\bigr)\leq R+1+r_1$. This contradicts the choice of $t_0$. Therefore, $\ell \cap D=\emptyset$. Let $\pi=c'_n\cup\ell\cup e(g*,v_{g_0P_0})$ then $\pi$ is an $(1,R+2)$--quasi-geodesic connecting $f(x_n)$, $f(x)$ and $\pi\cap D=\emptyset$. Therefore, $f(x_n)$ lies in $M_{(1, R+2)}\bigl(f(x),D\bigr)$.

We now assume that $x_n$ is a peripheral limit point. Then $f(x_n)=v_{gP}$, where $\alpha_n$ lies in some neighborhood of the peripheral left coset $gP$. In this case, we let $c'_n$ be a geodesic in $\hat{\Gamma}(G,S,\PP)$ connecting $\alpha_n(t_0)$ and $v_{gP}$. By Lemma \ref{l15}, all $G$--vertices of $c'_n$ lie in the $B$--neighborhood of $\alpha_n\bigl([t_0,\infty)\bigr)$ with respect to the metric $d_S$ and $\alpha_n\bigl([t_0,\infty)\bigr)$ lies in the $B$--neighborhood of the $c'_n$ with respect to the metric $d$. By using a similar argument as above, we can prove that $f(x_n)$ lies in $M_{(1, R+2)}\bigl(f(x),D\bigr)$.
\end{proof}

\begin{thm}Let $(G,\PP)$ be a finitely generated relatively hyperbolic group. Then there is a $G$--equivariant continuous map $f$ from the Morse boundary $\partial_M G$ to the Bowditch boundary $\partial (G,\PP)$ with the following properties:
\begin{enumerate}
\item The map $f$ maps the set of non-peripheral limit points of $\partial_M G$ injectively into the set of non-parabolic points of $\partial (G,\PP)$.
\item The map $f$ maps peripheral limit points of the same type in $\partial_M G$ to the same parabolic point in $\partial (G,\PP)$.
\end{enumerate}
In particular, if the Morse boundary of each peripheral subgroup is empty, then $f$ maps the Morse boundary $\partial_M G$ injectively into the set of non-parabolic points of $\partial (G,\PP)$.
\end{thm}

\begin{rem}
By the above theorem, the map $f$ constructed as above is injective if the Morse boundary of each peripheral subgroup is empty. However, the map $f$ is not a topological embedding in general even when the Morse boundary of each peripheral subgroup is empty. In fact, let $G$ be a finitely generated group with the presentation $G=\langle a,b,c| ab=ba\rangle$. Then the group $G$ is relatively hyperbolic with respect to the abelian subgroup $H$ generated by $a$ and $b$. Obviously, the Morse boundary of each peripheral subgroup $H$ is empty. 

Let $X_G$ be the Cayley complex of $G$ with respect to the above presentation and $\tilde{X}_G$ the universal covering space of $X_G$. It is well-known that $\tilde{X}_G$ can be equipped a $\CAT(0)$ metric and $G$ acts geometrically on $\tilde{X}_G$. For each $n$, let $\alpha_n$ be the geodesic in $\tilde{X}_G$ with initial point $e$ labeled by $c^na^nccccc\cdots$. Let $\alpha$ be a geodesic ray with initial point at $e$ labeled by $ccccc\cdots$. Obviously, the sequence $\{[\alpha_n]\}$ converges to $[\alpha]$ in the $\CAT(0)$ boundary of $\tilde{X}_G$. Therefore, the sequence $\{[\alpha_n]\}$ converges to $[\alpha]$ in the Bowditch boundary $\partial(G,\PP)$ (see Theorem 1.1 in \cite{MR3143594}). Thus, the set $F =\{[\alpha_n]\mid n\geq 1\}$ is not closed in $\partial(G,\PP)$. However, the set $F$ is closed in the Morse boundary $\partial_M G$ since $F \cap \partial_M^N \tilde{X}_G$ is finite for each $N$ in $\mathcal{M}$. This implies that $f$ is not a topological embedding. 
\end{rem}

\bibliographystyle{alpha}
\bibliography{Tran}
\end{document}